\pdfoutput=1
\documentclass[a4paper,11pt,pdf,reqno]{amsart}
\usepackage{enumerate, amsmath, amsfonts, amssymb, amsthm, thmtools, wasysym, graphics, graphicx, xcolor, frcursive,comment,bbm}
\usepackage[colorlinks=true,citecolor=cyan,backref=none]{hyperref}
\usepackage{caption}
\usepackage{subcaption} % Pour les sous-figures
\usepackage{lscape}
\usepackage{tikz-cd}
\usepackage{braids}
\usetikzlibrary{braids} 

\makeatletter
\newcommand{\thickhline}{%
    \noalign {\ifnum 0=`}\fi \hrule height 1pt
    \futurelet \reserved@a \@xhline
}

\definecolor{darkblue}{rgb}{0.0,0,0.7} 
 
\definecolor{darkred}{rgb}{0.7,0,0} 
\usepackage{hyperref}
\usepackage[all]{xy}
\usepackage{mathtools}
\usepackage[utf8]{inputenc}
\usepackage{utfsym}
\usepackage[T1]{fontenc}
\usepackage{stmaryrd}
\usepackage{adjustbox}

\usepackage{vmargin}            
\setmarginsrb{3.5cm}{2.2cm}{3.5cm}{2.2cm}{0cm}{0.6cm}{0cm}{0.3cm}

\usepackage{caption,lipsum}
\usepackage{graphicx}                  
\usepackage{pstricks,pst-plot,pst-text,pst-tree,pst-eps,pst-fill,pst-node,pst-math}
\usepackage{setspace}
\usepackage{multicol}

\usepackage{yhmath}
\hypersetup{
    citecolor=magenta,
    colorlinks=true,
    linkcolor=blue,
    filecolor=cyan,      
    urlcolor=magenta,
}

\numberwithin{equation}{section}

\newcommand{\Z}{\mathbb{Z}}

\newcommand{\R}{\mathbb{R}}
\newcommand{\llangle}{\langle\langle}
\newcommand{\rrangle}{\rangle\rangle}
\newcommand{\N}{\mathbb{N}}
\newcommand{\C}{\mathbb{C}}
\newcommand{\B}{\mathcal{B}}
\newcommand{\im}{\text{Im}}

\newcommand{\s}{\sigma}
\newcommand{\into}{\hookrightarrow}

\newcommand{\xto}{\xrightarrow}

\newcommand{\col}[6]{\left\{\!\!\!\left\{\begin{pmatrix}#1\\#2\end{pmatrix},\begin{pmatrix}#3\\#4\end{pmatrix},\begin{pmatrix}#5\\#6\end{pmatrix}\right\}\!\!\!\right\}}
\makeatletter
\newcommand*\bigcdot{\mathpalette\bigcdot@{.5}}
\newcommand*\bigcdot@[2]{\mathbin{\vcenter{\hbox{\scalebox{#2}{$\m@th#1\bullet$}}}}}
\makeatother
\theoremstyle{plain}

\newtheorem{theorem}{Theorem}[section]
\newtheorem{lemma}[theorem]{Lemma}
\newtheorem{proposition}[theorem]{Proposition}
\newtheorem{corollary}[theorem]{Corollary}

\theoremstyle{definition}

\newtheorem{definition}[theorem]{Definition}

\newtheorem{remark}[theorem]{Remark}

\newtheorem{notation}[theorem]{Notation}
\newtheorem{RemConj}[theorem]{Remark/Conjecture}

\title[$J-$braid groups are torus necklace groups]{$J-$braid groups are torus necklace groups}

\author{Igor Haladjian}
\address{Institut Denis Poisson, CNRS UMR 7019, Faculté des Sciences et Techniques, Université de Tours, Parc de Grandmont, 
37200 TOURS, France}
\DeclareRobustCommand{\SkipTocEntry}[5]{}
\begin{document}

%\tableofcontents
\thispagestyle{empty}

\begin{abstract}
We construct a family of links we call torus necklaces for which the link groups are precisely the braid groups of generalised $J$-reflection groups. Moreover, this correspondence exhibits the meridians of the aforementioned link groups as braid reflections. In particular, this construction generalises to all irreducible rank two complex reflection groups a well-known correspondence between some rank two complex braid groups and some torus knot groups. In addition, as abstract groups, we show that the family of link groups associated to Seifert links coincides with the family of circular groups. This shows that every time a link group has a non-trivial center, it is a Garside group.
\end{abstract}
\maketitle

\tableofcontents
\section{Introduction}
The 3-strand braid group $B_3$ is isomorphic to the knot group of the trefoil knot (also known as the $(2,3)$-torus knot). The aim of this article is to show that this isomorphism is part of a more general correspondence between some reflection groups and some link groups.\\
%This isomorphism links two a priori unrelated topics in mathematics, reflection groups and knot theory. It turns out that this isomorphism is part of a more general correspondence between some reflection groups and some link groups.\\
The 3-strand braid group belongs to the family of \textit{complex braid groups}. A \textit{complex reflection group} is a finite subgroup of $\mathrm{GL}_n(\C)$ ($n\in\N^*$) generated by complex reflections (a complex reflection is an element $r\in\mathrm{GL}_n(\C)$ of finite order whose set of fixed points is a complex hyperplane). Moreover, given a complex reflection group $W\subset\mathrm{GL}_n(\C)$, its \textit{rank} is defined as the dimension of the orthogonal complement of the subspace of $\C^n$ fixed pointwise by the action of $W$. See for example \cite{Broue} and \cite{LT} for a detailed exposition on complex reflection groups.\\
Given a complex reflection group $W$, one can define its complex braid group $\B_\C(W)$ as the fundamental group of the so-called \textit{regular orbit space} of $W$ (see \cite{Broue} and \cite{BMR} for a detailed exposition). For any complex reflection group $W$, there is a natural quotient $\B_\C(W)\to W$ (the group $W$ is realised as the group of deck transformations of a normal covering of its regular orbit space).
%Given a complex reflection group $W\subset \mathrm{GL}_n(\C)$, one defines its complex braid group $\B_\C(W)$ as $\pi_1(V^{\mathrm{reg}}/W)$, where $V^{\mathrm{reg}}$ is the complement in $\C^n$ of $\bigcup_{r\in R(W)}H_r$, where $R(W)$ is the set of reflections contained in $W$ (see \cite{BMR}).
Seeing the symmetric group $S_n$ as a real (hence complex) reflection group, the $n$-strand braid group $B_n$ is realised as $\B_\C(S_n)\cong B_n$ and the natural quotient $B_n\to S_n$ is the quotient which sends a braid diagram to the corresponding permutation diagram.\\
In \cite[Tables 1-4]{BMR}, Broué, Malle and Rouquier gave group presentations for all irreducible complex reflection groups, which we refer to as the BMR presentations. %In the case where $W$ is a real reflection group, its BMR presentation coincides with its Coxeter presentation.
They also conjectured that given a complex reflection group $W$ different from $G_{24}, G_{27}, G_{29}$, $G_{33}$ or $G_{34}$, removing the torsion relations of its BMR presentation gives a presentation for $\B_\C(W)$. This conjecture turned out to be true after the work of many mathematicians (see \cite{Bannai},\cite{BessisKP},\cite{BessisMichel},\cite{BrieskornF},\cite{BMR} and \cite{Garnier31}). It was also shown in \cite{BessisKP} and \cite{BMR} that the center of all complex braid groups is infinite cyclic and that for all complex reflection groups $W$ the image of $Z(\B_\C(W))$ under the natural quotient $\B_\C(W)\to W$ is $Z(W)$.

From now on, we will focus on a generalisation of rank two complex reflection groups and their braid groups. In \cite{AA}, Achar and Aubert defined $J$-groups (see Definition \ref{DefJGroups}), a family of groups generalising rank two complex reflection groups. The $J$-groups are defined as normal closure of elements in groups defined by generators and relations and while they appear as a natural generalisation of irreducible rank two complex reflection groups, their initial definition do not provide presentations by generators and relations. In fact, it is not known whether all $J$-groups are finitely generated or not.\\
The author then defined \textit{$J$-reflection groups} (see Definition \ref{DefJReflectionGroups}), a family of 
$J$-groups that includes all irreducible rank-two complex reflection groups which is easier to study than the whole family of $J$-groups. Specialising \cite[Theorem 1.2]{AA} to $J$-reflection groups, we have

\begin{theorem}
A group is isomorphic to a finite $J$-reflection group if and only if it is an irreducible complex reflection group of rank two.
\end{theorem}

In \cite{VCRG} the author gave a group presentation (see Presentation \eqref{GeneralPresW}) for any $J$-reflection group $W$, which coincides with its BMR presentation whenever $W$ is finite (hence a rank two complex reflection group). Mimicking the correspondence between the BMR presentation of a complex reflection groups and that of its associated braid group, one can define (combinatorially) the $J$-braid group $\B_J(W)$ of any $J$-reflection group $W$ by removing the torsion relations of Presentation \eqref{GeneralPresW} (see Presentation \eqref{BraidPres}). By construction, whenever $W$ is finite we have $\B_J(W)\cong\B_\C(W)$. \\
In this context, the conjugates of non-trivial powers of generators of Presentation \eqref{GeneralPresW} play a role of generalised reflections and coincide with actual reflections whenever $W$ is finite. Furthermore, the conjugates of non-trivial powers of generators of Presentation \eqref{BraidPres} play a role of braid reflections and coincide with actual braid reflections whenever $W$ is finite. One can then define \textit{reflecting hyperplanes} (see \cite[Definition 2.7]{Gobet Toric}). In \cite{Gobet Toric}, Gobet studied what he called \textit{toric reflection groups}, which are groups given by quotients of torus knot groups by (normal closures of) given powers of the meridians (see figure \ref{FigureMeridianIntro} for an illustration of meridians).
\begin{definition}[{\cite[Introduction]{Gobet Toric}}]
Let $n,m\in\N^*$ be coprime integers and $k\in\N_{\geq 2}$. Writing $\mathcal L(n,m)$ the knot group associated to the $(n,m)$-torus knot, define $W(k,n,m)$ as the quotient of $\mathcal L(n,m)$ by the normal closure of the $k$-th power of a preferred meridian.
\end{definition}

As such, the group $\mathcal L(n,m)$ can be defined (combinatorially) as the braid group of $W(k,n,m)$, the meridians of $\mathcal L(n,m)$ playing the role of braid reflections and their image by the quotient $\mathcal L(n,m)\to W(k,n,m)$ playing the role of reflections. Whenever $W(k,n,m)$ is finite it is always isomorphic to an irreducible rank two complex reflection group and the above procedure precisely describes its reflections, its braid group and the associated braid reflections. Moreover, Gobet showed that the family of toric reflection groups coincides with the family of $J$-reflection groups having exactly one conjugacy class of reflecting hyperplanes. In this article, we extend the correspondence to all $J$-reflection groups. More precisely, we exhibit links that we call \textit{torus necklaces} which will be playing the same role as torus knots do for toric reflection groups. \\

Alexander proved in \cite{Alexander} that every link is the closure of a braid. Moreover, in \cite{ArtinZopfe}, Artin introduced a right action of $B_n$ on the free group of rank $n$ and showed that it can be used to compute group presentations of link groups. More precisely, using Alexander's theorem every link can be written as the closure of some braid $\beta$ and we have the following:
\begin{theorem}[{\cite[Satz 6]{ArtinZopfe}, see also \cite[Theorem 2.2]{BirmanBook}}]\label{LinkGroupIntro}
Let $\beta \in B_n$. The link group $\pi_1(S^3\backslash \widehat{\beta})$ admits the presentation 
\begin{equation}\label{LinkPresIntro}
    \begin{aligned}
     \left\langle x_1,\dots,x_n\, | \, x_i=x_i\bigcdot \beta, \, \forall i\in [n] \right \rangle
    \end{aligned}.
\end{equation}
Moreover, for all $i\in [n]$ the generator $x_i$ represents the meridian which loops once around the $i$-th strand of $\beta$.
\end{theorem}

\begin{figure}[h]
    \centering
        \includegraphics[width=0.5\textwidth]{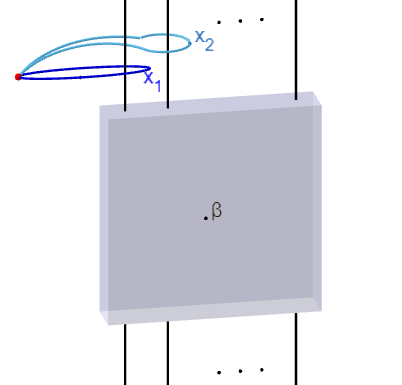}
        \caption{Generators of Presentation \eqref{LinkPresIntro} as meridians}
        \label{FigureMeridianIntro}
\end{figure}

The trefoil knot belongs to the family of torus knots. It turns out that whenever a knot group admits a non-trivial center, it is isotopic to a torus knot:

\begin{theorem}[{\cite{BZTorus}}]
Let $K$ be a knot. The knot group associated to $K$ has a non-trivial center if and only if $K$ is isotopic to a torus knot.
\end{theorem}

The classification of links whose link group have a non-trivial center was later generalised by Burde and Murasugi in \cite{SeifertLinks}. In order to state their result, we define the following links:
Let $n,m\in \N^*$. The link $\mathrm L_*^*(n,m)$ is the disjoint union of the $(n,m)$-torus link with a circle around the internal core of the torus and a circle around the external core of the torus (see Figure \ref{L23} for illustrations). Moreover, the link $\mathrm L_*(n,m)$ (respectively $\mathrm L^*(n,m)$) is the disjoint union of the $(n,m)$-torus link with a circle around the internal core (respectively around the external core) of the torus. Finally, the link $\mathrm L(n,m)$ is the $(n,m)$-torus link. A link of the form $\mathrm L_*^*(n,m),\mathrm L_*(n,m),\mathrm L^*(n,m)$ or $\mathrm L(n,m)$ is called a torus necklace. A link group associated to a torus necklace is then called a \textit{torus necklace group}.\\
In addition, Define $K(k)$ to be the key-chain link with $k+1$ connected components: the link consisting of $k$ disjoint circles around the  external core of the torus along with a circle around the internal core of the torus (see Figure \ref{KeyChainExamples} for an illustration). The classification result then goes as follows:

\begin{theorem}[{\cite[Theorem 1]{SeifertLinks}}]\label{BurdeMurasugiIntro}
Let $L$ be a link. The link group associated to $L$ has a non-trivial center if and only if $L$ is isotopic to a torus necklace or a key-chain link.
\end{theorem}

The links appearing in Theorem \ref{BurdeMurasugiIntro} are called \textit{Seifert links}, refering to the fact that they are precisely the links whose complement in $S^3$ is a Seifert fiber space. Not only these links are natural to consider as the only links whose link groups have non-trivial centers, they also are central objects in the study of the so-called $\mathrm{JSJ}$-decompositions of link complements. 
%Recall that given $M_1,M_2$ two $n$-dimensional topological manifold, one can construct their connected sum $M_1\#M_2$ by removing a ball from both $M_1$ and $M_2$ and gluing the resulting manifolds along the boundaries of these balls.
Recall that an $n$-manifold is \textit{prime} if it cannot be expressed as a non-trivial connected sum of $n$-manifolds. 
The $\mathrm{JSJ}$-decomposition theorem is a result proven in \cite{JS} and \cite{J} stating that any prime $3$-manifold $M$ contains a unique (up to isopoty) mimimal family $\{T_i\}_{i\in[k]}$ of embedded incompressible tori such that for all $i\in [k]$, the manifold $M\backslash T_i$ is a disjoint union of atoroidal manifolds (ones that do not contain any incompressible torus) and Seifert-fibered manifolds. When specialising to link complements in $S^3$, the Seifert links are those whose complements in $S^3$ are the Seifert-fibered manifolds arising in the $\mathrm{JSJ}$-decompositions of link complements. See \cite{Budney} for a detailed exposition of the $\mathrm{JSJ}$-decomposition in the case of link complements.\\ 

Thus, the family of $J$-braid groups associated to $J$-reflection groups with one conjugacy class of reflecting hyperplanes (that is, the toric reflection groups) coincides with a family of Seifert links, namely the torus knots. Using Theorem \ref{LinkGroupIntro} for computing group presentations of torus necklace groups on the one hand and comparing to the group presentations obtained by the author in \cite{BraidsJIgor} on the other hand, we extend the above result to all $J$-braid groups :
\begin{theorem}[Theorem \ref{NecklaceGroupsAreJBraidgroups}]\label{TheoremNecklaceBraidsIntro}
Let $W$ be a $J$-reflection group. There exists a torus necklace $L$ such that $\B_J(W)\cong \pi_1(\R^3\backslash L)$. 
Moreover, under this identification, the meridians of the torus necklace group correspond to the braid reflections of the $J$-braid group.
\end{theorem}

More generally, extending the definition of $J$-reflection groups to what we call generalised $J$-reflection groups, the same proof as that of Theorem \ref{TheoremNecklaceBraidsIntro} shows that we have a complete correspondence between the family of torus necklace groups and that of generalised $J$-braid groups, where meridians correspond to braid reflections (see Section \ref{SectionGeneralised}).\\

When specialising to finite $J$-reflection groups (equivalently, to irreducible complex reflection groups of rank two), combining the works of Bannai in \cite{Bannai} and Budney in \cite{Budney} allows to give a topological explanation to Theorem \ref{TheoremNecklaceBraidsIntro}:

\begin{remark}[{Corollary \ref{CorollaryHomotopyEquivalence}}]
Let $W$ be an irreducible complex reflection group of rank two and let $\mathrm L$ be the torus necklace corresponding to $\B_\C(W)$ in Table \ref{TableBraids}. There exists a homotopy equivalence between the regular orbit space of $W$ and $S^3\backslash \mathrm L$ which induces a bijection between (homotopy classes of) braid reflections and meridians. 
\end{remark}

Finally, given $n,m\in\N^*$, define $G(n,m)$ to be the group defined by the presentation $\langle a_1,\dots,a_n\,|\, a_1\cdots a_m=a_i\cdots a_{i+m-1},\, \forall \, i\in [n]\rangle$, where indices are taken modulo $n$. Such a group is called a \textit{circular group}. On the one hand, it is well known that complex braid groups of rank two are isomorphic to circular groups (see \cite[Theorems 1 and 2]{Bannai}). On the other hand, it is well known that torus knot groups are isomorphic to circular groups (see \cite{Gobet Garside} for a proof) and Bessis stated without proof in \cite{DualBessis} that all torus link groups are isomorphic to circular groups, with the specific case of $n$ dividing $m$ having been proved by Picantin in \cite{PicantinTorus}. In addition, it is known (see \cite{SeifertLinks}) that for any $n,m\in\N^*$, the torus necklace group $\pi_1(\R^3\backslash \mathrm L_*^*(n,m))$ is isomorphic to $G(n\wedge m+1,n\wedge m+1)$ and that the key-chain link group $\pi_1(\R^3\backslash K(k))$ is isomorphic to $G(k+1,k+1)$. Generalising some isomorphisms that where exhibited by the author in \cite{BraidsJIgor}, we generalise all these observations to obtain the following results:
\begin{theorem}[Theorems \ref{NecklaceGroupsAreJBraidgroups} and \ref{LnIso} and Remark \ref{CircularIso}]
Up to abstract isomorphisms, the families of torus necklace groups and circular groups are the same.
%(i) The group $\mathcal L_*^*(n,m)\cong\B_*^*(n,m)$ is isomorphic to $G(d+1,d+1)$.\\
%(ii) The group $\mathcal L_*(n,m)\cong\B_*(n,m)$ is isomorphic to $G(d+1,(d+1)m')$.\\
%(iii) The group $\mathcal L^*(n,m)\cong\B^*(n,m)$ is isomorphic to $G((d+1)n',d+1)$.\\
%(iv) The group $\mathcal L(n,m)\cong\B(n,m)$ is isomorphic to $G(n,m)$.\\
\end{theorem}
 Since circular groups are examples of Garside groups, we obtain the following corollary:

\begin{corollary}[Corollary \ref{GarsideLinkGroups}]\label{CorollaryGarside}
A group is isomorphic to a link group with non-trivial center if and only if it is isomorphic to a circular group. In particular, a link group is Garside if and only if it has a non-trivial center.
\end{corollary}

Note that every link group is bi-automatic (see \cite{LinksBiAutomatic}), and Corollary \ref{CorollaryGarside} provides an explicit bi-automatic structure for all Seifert link groups since Garside groups are bi-automatic (see \cite[Proposition 3.12]{DehornoyGarside}).

\addtocontents{toc}{\SkipTocEntry}
\section*{Acknowledgments}
This work is part of my PhD thesis, done under the supervision of Thomas Gobet and Cédric Lecouvey at Université de Tours, Institut Denis Poisson. 
I thank Thomas Gobet for his careful and numerous readings of all redaction stages of this article. I also thank him for all his precious advice, suggestions and his answers to my questions. I also thank Cédric Lecouvey for his careful reading at numerous stages of this article and his advice. Moreover, I thank Fathi Ben Aribi for many suggestions and informative discussions. Specifically, I thank him for pointing out \cite{AribiThesis} and \cite{Budney}. I also thank Michael Heusener for beneficial discussions. Finally, I thank Eléonore Meloni for her help in creating the figures of this article. This work was partially supported by the ANR
project CORTIPOM (ANR-21-CE40-0019).

\section{\texorpdfstring{Definitions and results about $J$-braid groups}{}}
\subsection{\texorpdfstring{$J$-reflection groups}{}}
The family of $J$-reflection groups was introduced by the author in \cite{VCRG}, as a subfamily of the family of $J$-groups introduced by Achar \& Aubert in \cite{AA}. The family of $J$-reflection group generalises another family of $J$-groups studied by Gobet in \cite{Gobet Toric} which he called toric reflection groups. In \cite{BraidsJIgor}, the author then defined the braid groups of $J$-reflection groups. The families of $J$-reflection groups and their braid groups generalise irreducible rank two complex reflection groups and their braid groups.\\

We start by defining $J$-groups and $J$-reflection groups.

\begin{definition}[{\cite[Introduction]{AA}}]\label{DefJGroups}
 For $k,n,m\in \N_{\geq 2}$, define $J\begin{pmatrix} k & n & m \\ & & \end{pmatrix}$ to be the group with presentation
 $$\langle s,t,u\, | \, s^k=t^n=u^m=1, \, stu=tus=ust\rangle.$$
More generally, for pairwise coprime elements $k',n',m'\in \N^*$ such that $x'$ divides $x$ for all $x\in \{k,n,m\}$, the group $J\begin{pmatrix} k & n & m \\ k'&n' &m' \end{pmatrix}$ is defined as the normal closure of $\{s^{k'},t^{n'},u^{m'}\}$ in $J\begin{pmatrix} k& n & m \\ & & \end{pmatrix}$. Finally, whenever $k',n'$ or $m'$ is equal to $1$, we omit it in the notation.
\end{definition}

\begin{definition}[{\cite[Definition 2.13]{VCRG}}]\label{DefJReflectionGroups}
Let $k,b,n,c,m\in\N^*$ with $k,bn,cm\geq 2$ and $n\wedge m=1$. The group $W_b^c(k,bn,cm)$ is defined as $J\begin{pmatrix} k & bn & cm \\ & n & m\end{pmatrix}$. Such a group is called a \textbf{$J$-reflection group}.
\end{definition}

Achar \& Aubert classified finite $J$-groups. The result, when specialised to $J$-reflection groups, takes the following form:
\begin{theorem}[{\cite[Theorem 1.2]{AA}}]
A group is a finite $J$-reflection group if and only if it is an irreducible complex reflection group of rank two.
\end{theorem}

A direct proof of the above result specialised to $J$-reflection groups was given in \cite{VCRG} (see Remark 2.44).

\begin{theorem}[{\cite[Theorem 2.29]{VCRG}}]\label{GeneralPres}
Let $k,b,n,c,m\in\N^*$ with $k,bn,cm\geq 2$ and $n\wedge m=1$. The group $W_b^c(k,bn,cm)$ admits the following presentation:
\begin{equation}\label{GeneralPresW}
\begin{aligned}
&(1) \,\, \mathrm{Generators}\!:\,  \{x_1,\dots,x_n,y,z\};\\
&(2) \,\, \mathrm{Relations}\!: \,\\
&x_i^k=y^b=z^c=1\, \forall i=1,\cdots,n, \\
  &  x_1\cdots x_nyz=zx_1\cdots x_ny,\\
       & x_{i+1}\cdots x_nyz\delta^{q-1}x_1\cdots x_{i+r}=x_i\cdots x_nyz\delta^{q-1}x_1\cdots x_{i+r-1}, \, \forall 1\leq i \leq n-r,\\
      &  x_{i+1}\cdots x_nyz\delta^qx_1\cdots x_{i+r-n}=x_i\cdots x_nyz\delta^qx_1\cdots x_{i+r-n-1},\, \forall n-r+1\leq i \leq n.
\end{aligned}
\end{equation}
where $\delta$ denotes $x_1\cdots x_ny$.
\end{theorem}

\begin{theorem}[{\cite[Theorem 2.52]{VCRG}}]\label{CenterFINDANAME}
The center of $W_b^c(k,bn,cm)$ is cyclic, generated by $(x_1\cdots x_ny)^mz^n$.
\end{theorem}

\begin{notation}
If $b$ or $c$ are equal to $1$, we omit them in the notation of $J$-reflection groups. More precisely, we denote the group $W_b^1(k,bn,m)$ by $W_b(k,bn,m)$, the group $W_1^c(k,n,cm)$ by $W^c(k,n,cm)$ and the group $W_1^1(k,n,m)$ by $W(k,n,m)$.
\end{notation}

\begin{remark}
(i) Toric reflection groups as defined by Gobet in \cite{Gobet Toric} are $J$-reflection groups, corresponding to the case $b=c=1$.\\
(ii) Every irreducible rank two complex reflection group is a $J$-reflection group (see \cite[Table 1]{VCRG}).
\end{remark}

It turns out that whenever $W_b^c(k,bn,cm)$ is finite, when seen as an irreducible complex reflection group of rank two, its complex reflections can be described as the conjugates of non-trivial powers of elements in the set $\{x_1,\dots,x_n,y,z\}$. More precisely, when $W_b^c(k,bn,cm)$ is finite, Presentation \eqref{GeneralPresW} coincides with its BMR presentation. This suggests to look at $J$-reflection groups as (abstract) reflection groups, as defined by Gobet:

\begin{definition}[{\cite[Definition 2.2]{Gobet Toric} and \cite[Lemma 2.24]{VCRG}}]\label{reflections}
The set of reflections of $W_b^c(k,bn,cm)$ is the set of conjugates of non-trivial powers of elements in the set $\{x_1,\dots,x_n,y,z\}$. We denote this set by $R(W_b^c(k,bn,cm))$.
\end{definition}

Definition \ref{reflections} suggests to study $J$-reflection groups not only as abstract groups but as groups with a distinguished set of elements generalising the set of reflections of complex reflection groups. More precisely, we have the following definition:

\begin{definition}[{\cite[Definition 2.3]{Gobet Toric}}]
Two $J$-reflection groups $H_1$ and $H_2$ are said to be \textbf{isomorphic in reflection} if there exists a group isomorphism $H_1\xto\varphi H_2$ such that $\varphi(R(H_1))=R(H_2)$.
\end{definition}

\begin{remark}[{\cite[Remark 2.4]{Gobet Toric}}]\label{isopermute}
The reflection isomorphism type of the group $J\begin{pmatrix} k & n & m\\ k'& n' &m' \end{pmatrix}$ is invariant under column permutations.
\end{remark}

In \cite[Theorem 1.2]{Gobet Toric}, Gobet classified toric reflection groups up to reflection isomorphisms, result which was then generalised by the author:

\begin{theorem}[{\cite[Theorem 3.11]{VCRG}}]
Let $k,k',b,b',n',n',c,c',m,m'\in \N^*$ and assume that 
$k,k',bn,b'n',cm,c'm'\geq 2$ and $n\wedge m=n'\wedge m'=1$. The $J$-reflection groups $W_b^c(k,bn,cm)$ and $W_{b'}^{c'}(k',b'n',c'm')$ are isomorphic in reflection if and only if the multisets $\col k1{bn}n{cm}m$ and $\col {k'}1{b'n'}{n'}{c'm'}{m'}$ are equal.
\end{theorem}

\subsection{\texorpdfstring{$J$-braid groups}{}}
The braid groups of irreducible complex reflection groups of rank two can be obtained combinatorially from their BMR presentation by removing the torsion of the generators (\cite{Bannai}). More generally, the above is true for any complex braid group (\cite{Bannai},\cite{BessisKP},\cite{BessisMichel},\cite{BrieskornF},\cite{BMR},\cite{Garnier31}). This theorem motivates the (combinatorial) definition of the braid group of any $J$-reflection group:

\begin{definition}[{\cite[Definitions 3.1-3.4]{BraidsJIgor}}]\label{DefBraid}
Let $n,m\in \N^*$ be two coprime integers and let $m=qn+r$ with $0\leq q$ and $0\leq r\leq n-1$. Let $\B_*^*(n,m)$ be the group defined by the following presentation:
\begin{subequations}\label{BraidPres}
\begin{align}
&(1) \,\, \mathrm{Generators}\!:\,  \{x_1,\dots,x_n,y,z\};\notag\\
&(2) \,\, \mathrm{Relations}\!: \,\notag\\
  &  x_1\cdots x_nyz=zx_1\cdots x_ny,\label{BraidPresDef:1}\\
       & x_{i+1}\cdots x_nyz\delta^{q-1}x_1\cdots x_{i+r}=x_i\cdots x_nyz\delta^{q-1}x_1\cdots x_{i+r-1}, \, \forall 1\leq i \leq n-r,\label{BraidPresDef:2}\\
      &  x_{i+1}\cdots x_nyz\delta^qx_1\cdots x_{i+r-n}=x_i\cdots x_nyz\delta^qx_1\cdots x_{i+r-n-1},\, \forall n-r+1\leq i \leq n,\label{BraidPresDef:3}
\end{align}
\end{subequations}
where $\delta$ denotes $x_1\cdots x_ny$.\\
Given $k,b,n,c,m\in\N^*$ such that $k,bn,cm\geq 2$, we now define the \textbf{braid group associated to} $W_b^c(k,bn,cm)$ as follows:\\
$\bullet$ If $b,c\geq 2$ we define it as $\B_*^*(n,m)$.\\
$\bullet$ If $b\geq 2$ and $c=1$ we define it as $\B_*(n,m):=\B_*^*(n,m)/\llangle z\rrangle$.\\
$\bullet$ If $b=1$ and $c\geq 2$ we define it as $\B^*(n,m):=\B_*^*(n,m)/\llangle y\rrangle$.\\
$\bullet$ If $b,c=1$ we define it as $\B(n,m):=\B_*^*(n,m)/\llangle y,z\rrangle$.\\
When useful, we refer to the above groups as $\B_J(W_b^c(k,bn,cm))$. Note that $\B(n,m)$ is already defined as $\B_J(W(k,n,m))$ in \cite[Introduction]{Gobet Toric}.\\

%which gives the following presentation:
%\begin{equation*}
%\begin{aligned}
%&(1) \,\, \mathrm{Generators}\!:\,  \{x_1,\dots,x_n,y\};\\
%&(2) \,\, \mathrm{Relations}\!: \,\\
%       & x_{i+1}\cdots x_ny\delta^{q-1}x_1\cdots x_{i+r}=x_i\cdots x_ny\delta^{q-1}x_1\cdots x_{i+r-1}, \, \forall 1\leq i \leq n-r,\\
 %     &  x_{i+1}\cdots x_ny\delta^qx_1\cdots x_{i+r-n}=x_i\cdots x_ny\delta^qx_1\cdots x_{i+r-n-1},\, \forall n-r+1\leq i \leq n.
%\end{aligned}
%\end{equation*}
%where $\delta$ denotes $x_1\cdots x_ny$.\\

%which gives the following presentation:
%\begin{equation*}
%\begin{aligned}
%&(1) \,\, \mathrm{Generators}\!:\,  \{x_1,\dots,x_n,z\};\\
%&(2) \,\, \mathrm{Relations}\!: \,\\
%  &  x_1\cdots x_nz=zx_1\cdots x_n,\\
%       & x_{i+1}\cdots x_nz\delta^{q-1}x_1\cdots x_{i+r}=x_i\cdots x_nz\delta^{q-1}x_1\cdots x_{i+r-1}, \, \forall 1\leq i \leq n-r,\\
 %     &  x_{i+1}\cdots x_nz\delta^qx_1\cdots x_{i+r-n}=x_i\cdots x_nz\delta^qx_1\cdots x_{i+r-n-1},\, \forall n-r+1\leq i \leq n.
%\end{aligned}
%\end{equation*}
%where $\delta$ denotes $x_1\cdots x_n$.\\

\end{definition}

\begin{remark}
In Definition \ref{DefBraid}, the braid groups of $J$-reflection groups are defined from specific presentations of $J$-reflection groups rather than from the group itself. Therefore, some classification results are needed to show that the braid group only depend on the reflection isomorphism type of $J$-braid groups (see Theorem \ref{WellDefBraids} below).
\end{remark}

\begin{remark}\label{SquareBraid}
Let $n,m\geq 2$ be two coprime integers. We have the following commutative square of quotients:

\begin{equation}\label{CommFINDANAMEbraids}
\xymatrix{\B_*^*(n,m) \ar[r]^{y=1}\ar[d]^{z=1} &\B^*(n,m)\ar[d]_{z=1} & \\
\B_*(n,m) \ar[r]^{y=1}& \B(n,m) & 
}\end{equation}
If $n=1$ (respectively $m=1$) we still have an epimorphism 
$\B_*^*(n,m)\xto{z=1} \B_*(n,m)$ (respectively $\B_*^*(n,m)\xto{y=1} \B^*(n,m)$), but the Square \eqref{CommFINDANAMEbraids} contains non-defined groups.
\end{remark}

We call \textbf{$J$-braid groups} the braid groups associated to $J$-reflection groups. In the case where $W_b^c(k,bn,cm)$ is finite, the presentation of its $J$-braid group coincides with the BMR presentation of its complex braid group. In particular, $J$-braid groups generalise rank two complex braid groups. Moreover, for any $J$-reflection group $W$, the natural quotient $\B_J(W)\to W$ sends conjugates of non-trivial powers of elements in $\{x_1,\dots,x_n,y,z\}$ to reflections of $W$. Thus, these elements can be seen as (generalised) braid reflections and we refer to them as such:
\begin{definition}
Let $W$ be a $J$-reflection group. The non-trivial powers of elements in $\{x_1,\dots,x_n,y,z\}\subset \B_J(W)$ are \textbf{braid reflections}.
\end{definition}

\begin{definition}[{\cite[Definition 2.1]{GarnierHoso}}]\label{CircularGroups}
Let $n,m\in \N^*$. The circular group $G(n,m)$ is defined by the presentation 
\begin{equation}\label{PresCircular}
\langle a_1,\dots,a_n\,|\, a_1\cdots a_m=a_2\cdots a_{m+1}=\cdots=a_n\cdots a_{n+m-1}\rangle,
\end{equation}
where indices are taken modulo $n$. Denote by $\alpha$ the element $a_1\cdots a_n$.
\end{definition}

These groups are Garside groups, as the corresponding monoids are Garside:

\begin{proposition}[{\cite[Example 5]{Origines}}]\label{CircularGarside}
Let $n,m\in\N^*$. The monoid with Presentation \eqref{PresCircular} is a Garside monoid.
\end{proposition}

We do not discuss Garside theory in detail in this article, and refer to \cite{GrosBouquinBleu} and \cite{Origines} for a precise exposition of the topic.
\begin{theorem}[{\cite[Corollary 2.11 ]{GarnierHoso}}]\label{CenterCircularGroups}
Let $n,m\in\N^*$. The center of $G(n,m)$ is cyclic, generated by $\alpha^{\frac m{n\wedge m}}$.
\end{theorem}

The following theorems show that $J$-braid groups are well defined up to reflection isomorphisms and that they are (abstractly) isomorphic to well-known groups: 
\begin{theorem}[{\cite[Theorem 3.18 and 3.26]{BraidsJIgor}}]\label{IsoBraidsJIgor}
Let $W:=W_b^c(k,bn,cm)$ be a $J$-reflection group. The following hold: \\
(i) If $b,c>1$, the group $\B_J(W)$ is isomorphic to $G(3,3)$.\\
(ii) If $b>1$ and $c=1$, the group $\B_J(W)$ is isomorphic to $G(2,2m)$.\\
(iii) If $b=1$ and $c>1$, the group $\B_J(W)$ is isomorphic to $G(2n,2)$.\\
(iv) If $b=c=1$, the group $\B_J(W)$ is isomorphic to $G(n,m)$.
\end{theorem}

\begin{theorem}[{\cite[Theorem 3.19, Corollary 3.29 and Proposition 3.37]{BraidsJIgor}}]\label{WellDefBraids}
Let $W_1$ and $W_2$ be $J$-reflection groups. The following hold:\\
(i) The center of $\B_J(W_1)$ is cyclic and sent onto the center of $W_1$ under the natural quotient $\B_J(W_1)\to W_1$.\\
(ii) If $W_1$ and $W_2$ are isomorphic in reflection, there exists an isomorphism between $\B_J(W_1)$ and $\B_J(W_2)$ sending braid reflections to braid reflections.
\end{theorem}

\section{Link groups}
In this section, we recall how to compute link groups using braid closures. Given a (geometric) braid $\beta\in B_n$, denote its closure by $\widehat{\beta}$. Recal the following celebrated theorem due to Alexander:
\begin{theorem}[{\cite{Alexander}, see also \cite[Theorem 2.3]{KasselTuraev}}]
Every link is the closure of some braid. 
\end{theorem}

Recall also that the $n$-strand braid group admits the presentation

\begin{equation*}\label{BraidPresentationArtinIntro}
  \left\langle 
\begin{array}{c}
    \sigma_1, \sigma_2, \dots, \sigma_{n-1} 
\end{array}
\;\middle|\;
\begin{array}{c}
 \sigma_i \sigma_{i+1} \sigma_i = \sigma_{i+1} \sigma_i \sigma_{i+1} \quad \text{for } 1 \leq i \leq n-2
   \\
     \sigma_i \sigma_j = \sigma_j \sigma_i \quad \text{for } |i - j| \geq 2
\end{array}
\right\rangle.
\end{equation*}

\begin{definition}\label{ArtinAct}
The braid group $B_n$ acts on the right on $F_n=\langle x_1,\dots,x_n\, |\, \emptyset \rangle$ in the following way: The generator $\s_i$ acts trivially on $\{x_1,\dots,x_n\}\backslash \{x_i,x_{i+1}\}$ and we have 
\begin{equation}\label{ArtinEq}
    \begin{aligned}
 \begin{cases} x_i\bigcdot \s_i&=x_ix_{i+1}x_i^{-1}, \\x_{i+1}\bigcdot \s_i&=x_i.\end{cases}
    \end{aligned}
\end{equation}
This action is called the \textbf{Artin representation} of $B_n$. Given a braid $\beta\in B_n$, we write $\rho_\beta$ for the automorphism of $F_n$ induced by $\beta$.
\end{definition}

\begin{remark}\label{StableProduct}
For all $i\in [n]$ we have $(x_ix_{i+1})\bigcdot \s_i=x_ix_{i+1}$, In particular, for $s\leq i\leq t$ we have $(x_s\cdots x_{t+1})\bigcdot \s_i=x_s\cdots x_{t+1}$.
\end{remark}

Recall that $\R^3$ can be seen as a subset of the 3-dimensional sphere $S^3$, and that given a link $L$, the map $\R^3\into S^3$ induces an isomorphism $\pi_1(\R^3\backslash L)\to \pi_1(S^3\backslash L)$.\\

We have the following theorem:
\begin{theorem}[{\cite[Satz 6]{ArtinZopfe}, \cite[Theorem 2.2]{BirmanBook}}]\label{LinkGroup}
Let $\beta \in B_n$. The link group $\pi_1(S^3\backslash \widehat{\beta})$ admits the presentation 
\begin{equation}\label{LinkPres}
    \begin{aligned}
     \left\langle x_1,\dots,x_n\, | \, x_i=\rho_\beta(x_i), \, \forall i\in [n] \right \rangle
    \end{aligned}.
\end{equation}
Moreover, for all $i\in [n]$ the generator $x_i$ represents the meridian which loops once around the $i$-th strand of $\beta$.  See Figure \ref{FigureMeridian} for an illustration.
\end{theorem}

\begin{figure}[h]
    \centering
        \includegraphics[width=0.5\textwidth]{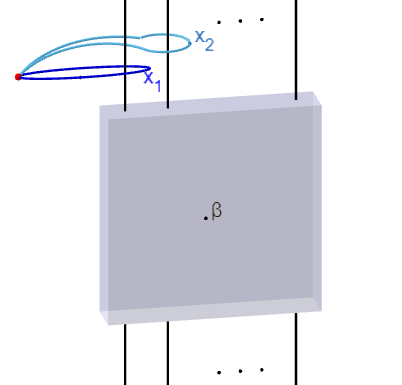}
        \caption{Generators of Presentation \eqref{LinkPres} as meridians}
        \label{FigureMeridian}
\end{figure}

\begin{remark}\label{AutoFree}
Given any group $G$ with presentation $\langle S\, |\, R\rangle$ and any automorphism $\varphi\in\mathrm{Aut}(F_S)$, the group $\varphi(G)$ with presentation $\langle S\, | \, \varphi(R)\rangle$ is isomorphic to $G$ via the morphism $G\to\varphi(G)$ sending $s$ to (the element represented by) $\varphi(s)$. In particular, Theorem \ref{LinkGroup} implies that for any braids $\beta,\s\in B_n$ the group $\pi_1(S^3\backslash \widehat{\beta})$ admits the presentation
\begin{equation}\label{ChangeBasis}
\langle x_1,\dots,x_n\,|\, \rho_\s(x_i)=\rho_{\beta\s}(x_i), \, \forall i\in [n]\rangle.
\end{equation}
In this case, the generators of Presentation \eqref{ChangeBasis} are still meridians.
\end{remark}

\begin{definition}\label{DefPure}
Given a partition $(I_1,\dots,I_k)$ of $[n]$, define the $(I_1,\dots,I_k)$-pure braid group as the subgroup of $B_n$ consisting of braids whose associated permutation stabilises $I_l$ for all $l\in [k]$. For $I=(I_1,\dots,I_k)$, the $I$-pure braid group is denoted as $P^I_n$. If all but one part of $I$ (say $I_k$) are singletons, we simply write $P^I_n$ as $P^J_n$, where $J=\cup_{l=1}^{k-1}I_l$.
\end{definition}

\begin{remark}
Given $I$ defined as in Definition \ref{DefPure}, for all $l\in [k]$ we have a quotient $q^I_l:P^I_n\to P^{(I_1,\dots,\widehat{I_l},\dots,I_k)}_{n-|I_l|}$ consisting of removing all strings whose starting index lies in $I_l$.
%We also denote by $p_l^I$ the quotient $F_n\to F_{n-|I_l|}$ which kills the generators with corresponding indices.
\end{remark}

A consequence of the Wirtinger algorithm (see e.g \cite[Chapter 3.D]{Rolfsen} for an exposition on the Wirtinger presentation) to determine a presentation of link groups is the following result:
\begin{proposition}\label{LinkQuotient}
Given $\beta\in B_n$, write $\pi(\beta)\in S_n$ for its associated permutation. Write $\textbf c=(c_1,\dots,c_k)$ the partition of $[n]$ given by the disjoint cycle decomposition of $\pi(\beta)$. For all $j\in [k]$, the groups $\pi_1(S^3\backslash \widehat{q_j^\textbf c(\beta)})$ and $\pi_1(S^3\backslash \widehat{\beta})/\llangle \{x_i\}_{i\in c_j}\rrangle$ are isomorphic, where the $x_i$'s are as in Presentation \eqref{LinkPres}. Moreover, the quotient $\pi_1(S^3\backslash\widehat{\beta})\to \pi_1(S^3\backslash \widehat{q_j^\textbf c(\beta)})$ sends meridians to meridians.
\end{proposition}

\section{Torus necklaces}
\subsection{Definition of torus necklaces}
The $J$-braid groups with one conjugacy class of reflecting hyperplanes are precisely the torus knot groups $\mathrm T(n,m)$, which can be seen as the knot group of the closure of the braid $(\s_1\cdots \s_{n-1})^m\in B_n$. We introduce a generalisation of torus knots, which will play the same role for $J$-braid groups having more than one conjugacy class of braid reflections. 

\begin{notation}
Denote by $q_{n+1}$ the quotient $P^{\{n+1,n+2\}}_{n+2}\to P^{\{n+1\}}_{n+1}$ which removes the second to last strand and by $q_{n+2}$ the quotient $P^{\{n+1,n+2\}}_{n+2}\to P^{\{n+1\}}_{n+1}$ which removes the last strand. Finally, denote by $q$ the quotient $P^{\{n+1,n+2\}}_{n+2}\to B_n$ which removes the last two strands. 
\end{notation}

\begin{definition}
Let $n,m\in \N^*$. Define $\mathrm L_*^*(n,m)$ to be the closure of the braid $\textbf{b}_{n,m}=(\s_{n+1}\cdots \s_1)(\s_1\cdots \s_{n+1})(\s_1\cdots \s_{n-1}\s_n^2)^m\in P^{\{n+1,n+2\}}_{n+2}$. Define $\mathrm L_*(n,m)$ (respectively $\mathrm L^*(n,m)$) to be the closure of $q_{n+2}(\textbf{b}_{n,m})$ (respectively $q_{n+1}(\textbf{b}_{n,m})$). Additionally, define $\mathrm L(n,m)$ to be the closure of $q(\textbf{b}_{n,m})$, which is simply the $(n,m)$-torus link. A link of one of the above families is called a \textbf{torus necklace}.  See Figures \ref{B23} and \ref{L23} for some illustrations.
\end{definition}

\begin{remark}
The case of torus links $\mathrm L(n,m)$ is a standard example. See for example \cite[Chapter 3.E]{BZH} or \cite[Chapter 3.C]{Rolfsen} for an exposition about torus knots. See also \cite{RZ} for a presentation of all torus link groups (Presentation \eqref{PresTorusLinkRZ}) and a proof of the meridional rank conjecture for the torus links.
\end{remark}

\begin{figure}[htb]
\begin{subfigure}{0.2\textwidth}
\centering
\begin{tikzpicture}
\braid[number of strands=5, height=0.25cm, width=0.45cm,strand 1/.style={red},
style strands={1,2,3}{cyan},
style strands={4}{magenta},
style strands={5}{orange}] (braid)a_4 a_3 a_2 a_1 a_1 a_2 a_3 a_4 a_1 a_2 a_3 a_3 a_1 a_2 a_3 a_3 a_1 a_2 a_3 a_3 a_1 a_2 a_3 a_3  ;
\end{tikzpicture}
\caption{$\textbf{b}_{3,4}$}
  \end{subfigure}%
  \begin{subfigure}{0.25\textwidth}
  \centering
\begin{tikzpicture}
\braid[number of strands=4, height=0.4cm, width=0.6cm,
style strands={1,2,3}{cyan},
style strands={4}{orange}] (braid)a_3 a_2 a_1 a_1 a_2 a_3 a_1 a_2 a_1 a_2 a_1 a_2 a_1 a_2  ;

\end{tikzpicture}
 \caption{$q_4(\textbf{b}_{3,4})$}
  \end{subfigure}%
  \begin{subfigure}{0.2\textwidth}
  \centering
\begin{tikzpicture}
\braid[number of strands=4, height=0.35cm, width=0.5cm,
style strands={1,2,3}{cyan},
style strands={4}{magenta}] (braid) 
a_1 a_2 a_3 a_3 a_1 a_2 a_3 a_3 a_1 a_2 a_3 a_3 a_1 a_2 a_3 a_3 ;
\end{tikzpicture}
 \caption{$q_5(\textbf{b}_{3,4})$}
  \end{subfigure}%
\begin{subfigure}{0.2\textwidth}
\centering
\begin{tikzpicture}
\braid[number of strands=3, height=0.7cm, width=0.9cm,
style strands={1,2,3}{cyan}] (braid) 
a_1 a_2 a_1 a_2 a_1 a_2 a_1 a_2 ;
\end{tikzpicture}
 \caption{$q(\textbf{b}_{3,4})$}
  \end{subfigure} 
  
\caption{The braids $\textbf{b}_{3,4}$, $q_2(\textbf{b}_{3,4})$, $q_5(\textbf{b}_{3,4})$ and $q(\textbf{b}_{3,4})$ }
\label{B23}
\end{figure}

\begin{figure}[h]
    \centering
    \begin{subfigure}{0.4\textwidth} % Ajuste la largeur si nécessaire
        \centering
        \includegraphics[width=\textwidth]{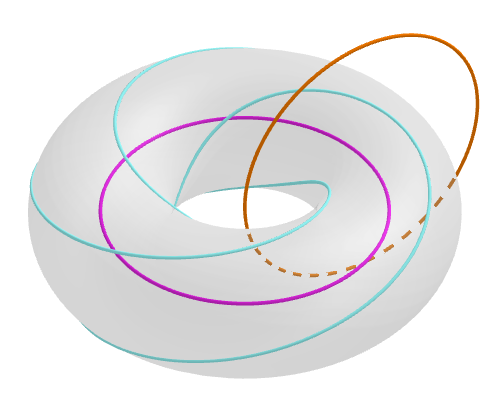}
        \caption{$\mathrm L_*^*(3,4)$}
        \label{fig:sousfig1}
    \end{subfigure}
        \begin{subfigure}{0.4\textwidth} % Ajuste la largeur si nécessaire
        \centering
        \includegraphics[width=\textwidth]{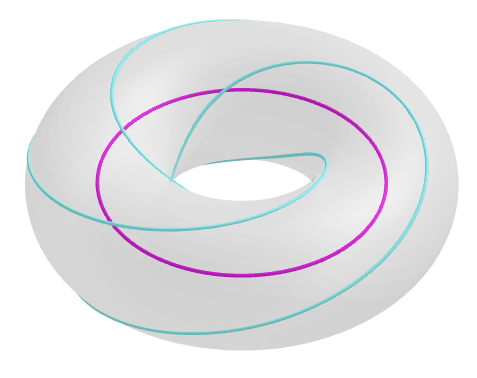}
        \caption{$\mathrm L_*(3,4)$}
        \label{fig:sousfig2}
    \end{subfigure}
        \begin{subfigure}{0.4\textwidth} % Ajuste la largeur si nécessaire
        \centering
        \includegraphics[width=\textwidth]{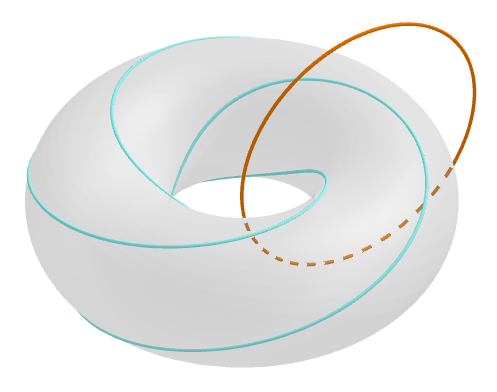}
        \caption{$\mathrm L^*(3,4)$}
        \label{fig:sousfig3}
    \end{subfigure}
     \begin{subfigure}{0.38\textwidth} % Ajuste la largeur si nécessaire
        \centering
        \includegraphics[width=\textwidth]{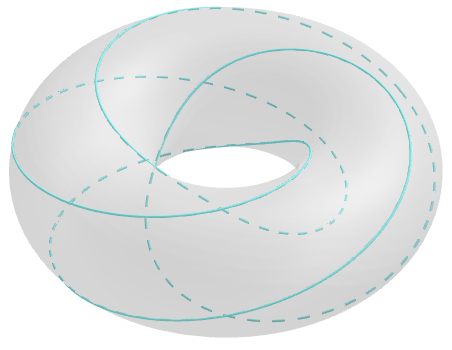}
        \caption{$\mathrm L(3,4)$}
        \label{fig:sousfig4}
    \end{subfigure}
\caption{The torus necklaces $\mathrm L_*^*(3,4)$, $\mathrm L_*(3,4)$, $\mathrm L^*(3,4)$ and $\mathrm L(3,4)$}
\label{L23}
\end{figure}

Seifert links are the links whose complement in $S^3$ are Seifert fiber space. 
They are natural to consider, as they play a central role in JSJ-decompositions of link complements (see \cite{JS},\cite{J} for an exposition on JSJ-decompositions of 3-manifolds and \cite{Budney} for the specific study in the case of link complements). It turns out that the family of torus necklaces almost coincides with the family of Seifert links. Indeed, it is shown in \cite[Theorem 1]{SeifertLinks} (see also \cite[Proposition 3.3]{Budney}) that the only Seifert links which are not torus necklaces are the key-chain links:
\begin{definition}\label{DefKeyChainLinks}
Let $k\in\N$. The key-chain link $\mathrm{K}(k)$ is the closure of the braid $\s_k\cdots\s_2s_1^2\s_2\cdots \s_k\in B_{k+1}$. See Figure \ref{KeyChainExamples} for an illustration.
\end{definition}

\begin{figure}
\begin{subfigure}{0.4\textwidth}
\centering
\begin{tikzpicture}
\braid[number of strands=6, height=0.40cm, width=0.55cm,strand 1/.style={red},
style strands={1,2,3,4,5}{cyan},
style strands={6}{orange},
] (braid)a_5 a_4 a_3 a_2 a_1 a_1 a_2 a_3 a_4 a_5;
\end{tikzpicture}
\caption{$(\s_5\s_4\s_3\s_2\s_1)(\s_1\s_2\s_3\s_4\s_5)$}
  \end{subfigure}%
 \begin{subfigure}{0.55\textwidth} % Ajuste la largeur si nécessaire
        \centering
        \includegraphics[width=\textwidth]{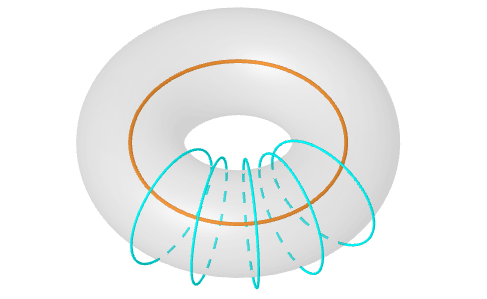}
        \caption{$\mathrm K(5)$}
        \label{fig:sousfig5}
    \end{subfigure}
 \caption{The braid $(\s_5\s_4\s_3\s_2\s_1)(\s_1\s_2\s_3\s_4\s_5)$ and the key-chain link $\mathrm{K}(5)$}
 \label{KeyChainExamples}
\end{figure}

\subsection{Link groups of torus necklaces}
This section is devoted to provide a convenient group presentation for the link group of a given torus necklace where generators are meridian. This is achieved by computing the Artin representation of $\textbf{b}_{n,m}$ and applying Theorem \ref{LinkGroup}.

\begin{definition}
Let $n,m\in \N^*$. Define the group $\mathcal L_*^*(n,m)$ (respectively $\mathcal L_*(n,m)$, $\mathcal L^*(n,m)$ and $\mathcal L(n,m)$) to be the link group of $\mathrm L_*^*(n,m)$ (respectively $\mathrm L_*(n,m)$, $\mathrm L^*(n,m)$ and $\mathrm L(n,m)$). We call such groups \textbf{torus necklace groups}.
\end{definition}

Recall that the links $\mathrm L(n,m)$ and $\mathrm L(m,n)$ are isotopic (see e.g \cite[Proposition 7.2.3 and Theorem 7.2.4]{Book Murasugi}). In \cite{Budney}, Budney extends this correspondence and completely determines when two torus necklaces are isotopic. In particular, he shows the following:
\begin{proposition}[{\cite[Proposition 3.5]{Budney}}]\label{IsotopicLinks}
Let $n,m\in\N^*$. The following hold:\\
(i) The links $\mathrm L_*^*(n,m)$ and $\mathrm L_*^*(m,n)$ are isotopic.\\
(ii) The links $\mathrm L_*(n,m)$ and $\mathrm L^*(m,n)$ are isotopic.\\
(iii) The links $\mathrm L(n,m)$ and $\mathrm L(m,n)$ are isotopic.\\
In particular, the corresponding link groups are isomorphic.
\end{proposition}

The family of torus necklace groups has been studied at various points in the literature. For example, it is well known (e.g \cite[Proposition 3.38]{BZH}) that if $n$ and $m$ are coprime, the torus knot group $\mathcal L(n,m)$ admits the group presentation \begin{equation}\label{PresTorusKnot}
    \langle x,y\,|\, x^n=y^m\rangle.
\end{equation}
More generally, for any $n,m\in\N^*$ with $d=n\wedge m$, $n=dn'$, $m=dn'$ and $r,s\in\Z$ such that $nr+ms=1$, \cite[Lemma 2.2]{RZ} shows that the torus link group $\mathcal L(n,m)$ admits the group presentation 
\begin{equation}\label{PresTorusLinkRZ}
    \langle x,y,m_1,\dots,m_d\,|\,x^{n'}=y^{m'},\, m_1\cdots m_d=x^sy^r,\, m_ix^{n'}=x^{n'}m_i \, \forall i=1,\dots,d\rangle.
\end{equation}

Moreover, Ben Aribi has shown during his thesis (see \cite[Theorem 4.9 and Lemma A.2]{AribiThesis}) that for $n$ and $m$ coprime the torus necklace group $\mathcal L_*^*(n,m)$ (respectively $\mathcal L_*(n,m)\cong \mathcal L^*(m,n)$) admits the group presentation
\begin{equation}\label{PresAribiNecklace1}
    \langle x,y,z,t\,|\,x^nz^m=y^mt^n,\, xz=zx,\, yt=ty\rangle
\end{equation}
\begin{equation}\label{PresAribiNecklace2}
    (\text{respectively } \langle x,y,t\,|\, x^n=y^mt^n, \, yt=ty\rangle).
\end{equation}
It is also known that if $n$ and $m$ are coprime, the torus knot group $\mathcal L(n,m)$ is isomorphic to $G(n,m)$ (see for example \cite[Lemma 3.1, Proposition 4.12 and Remark 4.13]{Gobet Garside} for a proof).
Finally, when $n$ divides $m$ Picantin has shown (see \cite{PicantinTorus}) that $\mathcal L(n,m)\cong G(n,m)$. We shall see that this result holds in general (see Theorem \ref{NecklaceGroupsAreJBraidgroups} and Remark \ref{CircularIso}).\\

The main result of this section is the following:

\begin{theorem}\label{NecklaceGroupsAreJBraidgroups}
Let $n,m\in \N^*$ be coprime integers. We have the following:\\
(i) The group $\mathcal L_*^*(n,m)$ is isomorphic to $\B_*^*(n,m)$.\\
(ii) If $n\geq 2$, the group $\mathcal L_*(n,m)$ is isomorphic to $\B_*(n,m)$.\\
(iii) If $m\geq 2$, the group $\mathcal L^*(n,m)$ is isomorphic to $\B^*(n,m)$.\\
(iv) If $n,m\geq 2$, the group $\mathcal L(n,m)$ is isomorphic to $\B(n,m)$.\\
Moreover, under this identification, meridians of the link groups correspond to braid reflections of the $J$-braid groups.
\end{theorem}

We will split the computation of the Artin representation of $\mathbf{b}_{n,m}$ in several Lemmas, and we encompass all the braids whose Artin representation will be computed in the following definition:
\begin{definition}\label{DefBraidsProofs}
Let $n\in\N^*$. Define the following:\\
(i) For $1\leq s\leq t\leq n-1$, define $U_{s,t}$ to be $\s_s\s_{s+1}\cdots\s_t\in B_n$.\\
(ii) For $1\leq s\leq t\leq n-1$, define $D_{s,t}$ to be $\s_t\s_{t-1}\cdots\s_t \in B_n$.\\ 
(iii) Define $V_n$ to be $D_{1,n+1}U_{1,n+1}\in B_{n+2}$.\\
(iv) Define $\beta_n$ to be $U_{1,n}\s_n\in B_{n+1}$.\\
(v) Define $\Delta_n$ to be $U_{1,n-1}U_{1,n-2}\cdots U_{1,1}\in B_n$.
\end{definition}

\begin{remark}
For all $n,m\in \N^*$, we have $\mathbf{b}_{n,m}=V_n\beta_n^m$.
\end{remark}

\begin{notation}
Given a group $G$ and $g,h\in G$, write $^{(h)}{g}$ for $hgh^{-1}$. 
\end{notation}

\begin{lemma}\label{ComputeArtin1}
Let $s,t\in [n-1]$ and assume that $s\leq t$. Recall from Definition \ref{DefBraidsProofs} that $U_{s,t}$ is the element $\s_s\s_{s+1}\cdots \s_t$. We have
\begin{equation}
    \begin{aligned}
     \rho_{U_{s,t}}(x_i)=\begin{cases}
        x_i \text{  $\mathrm{if}$ $i\notin \llbracket s,t+1\rrbracket$,}\\
        (x_s\cdots x_t)x_{t+1}(x_s\cdots x_t)^{-1} \, \, \text{$\mathrm{if}$ $i=s$,}\\
        x_{i-1} \text{  $\mathrm{if}$ $i\in \llbracket s+1,t+1\rrbracket$.}
     \end{cases}
    \end{aligned}
\end{equation}
\end{lemma}

\begin{proof}
 If $i\notin \llbracket s,t+1\rrbracket$, the action of $U_{s,t}$ on $x_i$ is trivial. Let $i\in \llbracket s+1,t+1\rrbracket$. We have
\begin{equation*}
     \begin{aligned}
      \rho_{U_{s,t}}(x_i)&=x_i\bigcdot(\s_{i-1}\cdots \s_t)\\
      &=x_{i-1}\bigcdot (\s_{i}\cdots \s_t)\\
      &=x_{i-1}.
     \end{aligned}
\end{equation*}
Moreover, we have
\begin{equation*}
    \begin{aligned}
     \rho_{U_{s,t}}(x_s)=&(x_sx_{s+1}x_s^{-1})\bigcdot(\s_{s+1}\cdots \s_t)\\
     &\vdots\\
     =&(x_s\cdots x_t)x_{t+1}(x_s\cdots x_t)^{-1}.
    \end{aligned}
\end{equation*}
This concludes the proof.
\end{proof}

\begin{lemma}\label{ComputeArtin2}
Let $s,t\in [n-1]$ and assume that $s\leq t$. Recall from Definition \ref{DefBraidsProofs} that $D_{s,t}$ is the element $\s_t\s_{t-1}\cdots \s_s$. We have
\begin{equation}
    \begin{aligned}
     \rho_{D_{s,t}}(x_i)=\begin{cases}
        x_i \, \, \text{$\mathrm{if}$ $i\notin \llbracket s,t+1\rrbracket$,}\\
        x_sx_{i+1}x_s^{-1} \, \, \text{$\mathrm{if}$ $i\in \llbracket s,t\rrbracket$,}\\
        x_s\, \, \text{$\mathrm{if}$ $i=t+1$}.\\
     \end{cases}
    \end{aligned}
\end{equation}
In particular, we have $\rho_{D_{s,t}}(x_s\cdots x_{t+1})=x_s\cdots x_{t+1}$.
\end{lemma}

\begin{proof}
 If $i\notin \llbracket s,t+1\rrbracket$, the action of $D_{s,t}$ on $x_i$ is trivial. Let $i\in \llbracket s,t\rrbracket$. We have
\begin{equation*}
     \begin{aligned}
      \rho_{D_{s,t}}(x_i)&=x_i\bigcdot(\s_i\cdots \s_s)\\
      &=(x_ix_{i+1}x_i^{-1})\bigcdot(\s_{i-1}\cdots \s_s)\\
      &=x_sx_{i+1}x_s^{-1}.
     \end{aligned}
\end{equation*}
Finally, we have $\rho_{D_{s,t}}(x_{t+1})=\rho_{D_{s,t-1}}(x_t)=\cdots=x_s$, which concludes the proof.
\end{proof}

\begin{lemma}\label{ArtinToric1}
Let $n\in \N^*$. Recall from Definition \ref{DefBraidsProofs} that the element $V_n$ is $D_{1,n+1}U_{1,n+1}$. We have

\begin{equation}
    \begin{aligned}
     \rho_{V_n}(x_i)=\begin{cases} 
     {}^{({}^{(x_1\cdots x_{n+1})}x_{n+2})}x_i \text{  $\mathrm{if}$ $i\in [n+1]$,}\\
     {}^{(x_1\cdots x_{n+1})}x_{n+2} \text{  $\mathrm{if}$ $i=n+2$.}
     \end{cases}
    \end{aligned}
\end{equation}

\end{lemma}

\begin{proof}
We have

\begin{equation}
    \begin{aligned}
     \rho_{V_n}(x_i)&=\rho_{U_{1,n+1}}\circ\rho_{D_{1,n+1}}(x_i)\\
     &=\begin{cases} 
     \rho_{U_{1,n+1}}({}^{(x_1)}x_{i+1}) \text{  $\mathrm{if}$ $i\in [n+1]$,}\\
     \rho_{U_{1,n+1}}(x_1) \text{  $\mathrm{if}$ $i=n+2$.}
     \end{cases}\text{ By Lemma \ref{ComputeArtin2},}\\
     &=\begin{cases} 
     {}^{({}^{(x_1\cdots x_{n+1})}x_{n+2})}x_i \text{  $\mathrm{if}$ $i\in [n+1]$,}\\
     {}^{(x_1\cdots x_{n+1})}x_{n+2} \text{  $\mathrm{if}$ $i=n+2$.}
     \end{cases}\text{ By Lemma \ref{ComputeArtin1}.}
    \end{aligned} 
\end{equation}

\end{proof}

\begin{lemma}\label{ArtinToric2}
Let $n\in \N^*$. Recall from Definition \ref{DefBraidsProofs} that $\beta_n$ is the element $U_{1,n}\s_n$. For all $q\in \N^*$ and for all $1\leq r\leq n-1$, we have
\begin{equation}
    \begin{aligned}
     \rho_{\beta_n^{qn+r}}(x_i)=\begin{cases} 
     {}^{((x_1\cdots x_{n+1})^{q+1})}x_{n-r+i} \text{  $\mathrm{if}$ $1\leq i\leq r$,}\\
     {}^{((x_1\cdots x_{n+1})^q)}x_{i-r} \text{  $\mathrm{if}$ $r+1\leq i \leq n$,}\\
     {}^{((x_1\cdots x_{n+1})^qx_{n-r+1}\cdots x_n)}x_{n+1} \text{  $\mathrm{if}$ $i=n+1$,}\\
     x_{n+2} \text{  $\mathrm{if}$ $i=n+2$.}
     \end{cases}
    \end{aligned}
\end{equation}
\end{lemma}

\begin{proof}
We have 
\begin{equation}
    \begin{aligned}
    \rho_{\beta_n}(x_i)&=\rho_{\s_n}\circ \rho_{U_{1,n}}(x_i)\\
    &=\begin{cases}
       \rho_{\s_n}({}^{(x_1\cdots x_n)}x_{n+1}) \text{   $\mathrm{if}$ $i=1$,}\\ 
       \rho_{\s_n}(x_{i-1}) \text{   $\mathrm{if}$ $i\in \llbracket 2,n+1\rrbracket$,}\\
       \rho_{\s_n}(x_{n+2}) \text{  if $i=n+2$.}
    \end{cases}\text{ By Lemma \ref{ComputeArtin1},}\\
    &=\begin{cases}
         {}^{(x_1\cdots x_{n+1})}x_n \text{  if $i=1$,}\\
         x_{i-1} \text{  if $i\in \llbracket 2,n\rrbracket$,}\\
         {}^{(x_nx_{n+1})}x_{n+1} \text{  if $i=n+1$,}\\
         x_{n+2} \text{  if $i=n+2$.}
    \end{cases}
    \end{aligned}
\end{equation}
Recall by Remark \ref{StableProduct} that $x_1\dots x_{n+1}$ is fixed by $\beta_n$.
Thus, by induction, for all $1\leq r \leq n$ we have 
\begin{equation}\label{ArtinToric2Eq2}
\begin{aligned}
    \rho_{\beta_n^r}(x_i)=\begin{cases}
       {}^{(x_1\cdots x_{n+1})}x_{n-r+i} \text{   $\mathrm{if}$ $1\leq i\leq r$,}\\ 
       x_{i-r} \text{   $\mathrm{if}$ $r+1\leq i\leq n$,}\\
       {}^{(x_{n-r+1}\cdots x_{n+1})}x_{n+1} \text{  if $i=n+1$,}\\
       x_{n+2} \text{  if $i=n+2$.}
    \end{cases}
\end{aligned}
\end{equation}
In particular, Equation \eqref{ArtinToric2Eq2} with $r=n$ gives 
\begin{equation}\label{ArtinToric2Eq3}
\begin{aligned}
    \rho_{\beta_n^n}(x_i)=\begin{cases}
       {}^{(x_1\cdots x_{n+1})}x_{i} \text{   $\mathrm{if}$ $1\leq i\leq n$,}\\
       {}^{(x_1\cdots x_{n+1})}x_{n+1} \text{  if $i=n+1$,}\\
       x_{n+2} \text{  if $i=n+2$.}
    \end{cases}
\end{aligned}
\end{equation}
Again using that $x_1\cdots x_{n+1}$ is fixed by $\beta_n$, we get

\begin{equation}
\begin{aligned}
    \rho_{\beta_n^{qn+r}}(x_i)&=\rho_{\beta_n^r}\circ\rho_{\beta_n^{qn}}(x_i)\\
    &=\begin{cases}
       \rho_{\beta_n^r}({}^{((x_1\cdots x_{n+1})^{q})}x_{i}) \text{   $\mathrm{if}$ $1\leq i\leq n$,}\\
       \rho_{\beta_n^r}({}^{((x_1\cdots x_{n+1})^q)}x_{n+1}) \text{  if $i=n+1$,}\\
       \rho_{\beta_n^r}(x_{n+2}) \text{  if $i=n+2$.}
    \end{cases} \text{by Equation \eqref{ArtinToric2Eq3},}\\
    &=\begin{cases}
       {}^{((x_1\cdots x_{n+1})^{q+1})}x_{n-r+i} \text{   $\mathrm{if}$ $1\leq i\leq r$,}\\
       {}^{((x_1\cdots x_{n+1})^{q})}x_{i-r} \text{   $\mathrm{if}$ $r+1\leq i\leq n$,}\\
       {}^{((x_1\cdots x_{n+1})^q x_{n-r+1}\cdots x_{n+1})}x_{n+1} \text{  if $i=n+1$,}\\
       x_{n+2} \text{  if $i=n+2$.}
    \end{cases} \text{by Equation \eqref{ArtinToric2Eq2},}
\end{aligned}
\end{equation}
which concludes the proof.
\end{proof}

We get the following:
\begin{proposition}\label{ArtinToric}
Let $n,m\in \N^*$ and write $m=qn+r$ the euclidean division of $m$ by $n$. The action of $\textbf{b}_{n,m}$ on $\{x_1,\dots,x_{n+2}\}$ is given by 
\begin{equation}
    \begin{aligned}
     \rho_{\textbf b_{n,m}}(x_i)=\begin{cases} {}^{(x_1\cdots x_{n+2}(x_1\cdots x_{n+1})^{q})}x_{n-r+i}\text{  $\mathrm{if}$ $1\leq i\leq r$,}\\ 
     {}^{(x_1\cdots x_{n+2}(x_1\cdots x_{n+1})^{q-1})}x_{i-r}\text{  $\mathrm{if}$ $r+1\leq i\leq n$,}\\
     {}^{(x_1\cdots x_{n+2}(x_1\cdots x_{n+1})^{q-1}x_{n-r+1}\cdots x_n)}x_{n+1} \text{  $\mathrm{if}$ $i=n+1$,}\\
     {}^{(x_1\cdots x_{n+1})}x_{n+2} \text{  $\mathrm{if}$ $i=n+2$.}
     \end{cases}
    \end{aligned}
\end{equation}
\end{proposition}

\begin{proof}
Since $\textbf{b}_{n,m}=V_n\beta_n^m$, the result follows by combining Lemmas \ref{ArtinToric1} and \ref{ArtinToric2}.
\end{proof}

\begin{lemma}\label{ArtinDelta}
Let $n\in \N^*$. Recall from Definition \ref{DefBraidsProofs} that $\Delta_n$ is the element $U_{1,n-1}U_{1,n-2}\cdots U_{1,1}$. Seeing $\Delta_n$ as an element of $B_{n+2}$ via the injection $B_n\to B_{n+2}$, the action of $\Delta_n$ on $\{x_1,\dots,x_{n+2}\}$ is given by
\begin{equation}\label{ArtinDelta1}
    \begin{aligned}
     \rho_{\Delta_n}(x_i)=\begin{cases} 
     {}^{(x_1\cdots x_{n-i})}x_{n-i+1} \text{  $\mathrm{if}$ $i\in [n]$,}\\
     x_i \text{  $\mathrm{if}$ $i\in \{n+1,n+2\}$.}
     \end{cases}
    \end{aligned}
\end{equation}
Moreover, for all $i\in [n]$, we have 
\begin{equation}\label{ArtinDelta2}
    \begin{aligned}
    \rho_{\Delta_n}(x_i\cdots x_n)=x_1\cdots x_{n-i+1}.
    \end{aligned}
\end{equation}
\end{lemma}

\begin{proof}
We show that Equation \eqref{ArtinDelta1} holds by induction on $n$. For $n=1$, we have $\Delta_n=1$ and the result is trivial. Assume that \eqref{ArtinDelta1} holds for $n\in \N^*$. Then, extending the action of $\Delta_n$ to $\{x_1,\dots,x_{n+3}\}$ via the injection $B_{n+2}\to B_{n+3}$, we have
\begin{equation}
    \begin{aligned}
    \rho_{\Delta_{n+1}}(x_i)&=\rho_{\Delta_n}\circ\rho_{U_{1,n}}(x_i)\\
    &=\begin{cases}
        \rho_{\Delta_n}({}^{(x_1\cdots x_n)}x_{n+1})\text{  if $i=1$,}\\
        \rho_{\Delta_n}(x_{i-1}) \text{  if $i\in\llbracket 2,n+1\rrbracket$,}\\
        \rho_{\Delta_n}(x_i) \text{  if $i\in\{n+2,n+3\}$}
    \end{cases}\text{ by Lemma \ref{ComputeArtin1},}\\
    &=\begin{cases}
    {}^{(x_1\cdots x_n)}x_{n+1} \text{  $\mathrm{if}$ $i=1$,}\\
    {}^{(x_1\cdots x_{n-i+1})}x_{n-i+2} \text{  $\mathrm{if}$ $i\in \llbracket 2,n+1\rrbracket$,}\\
     x_i \text{  $\mathrm{if}$ $i\in \{n+2,n+3\}$.}
     \end{cases} \text{ by induction hypothesis,}
    \end{aligned}
\end{equation}
where in the last equality for $i=1$ we used Remark \ref{StableProduct} to see that $x_1\cdots x_n$ is fixed by $\Delta_n$. This proves \eqref{ArtinDelta1}. For all $i\in [n]$, we then have 
\begin{equation}
     \begin{aligned}
     \rho_{\Delta_n}(x_i\cdots x_n)&=\prod_{j=i}^{n}(x_1\cdots x_{n-j})x_{n-j+1}(x_1\cdots x_{n-j})^{-1}\\
     &=x_1\cdots x_{n-i+1},
     \end{aligned}
\end{equation}
which concludes the proof.
\end{proof}

We are now ready to prove Theorem \ref{NecklaceGroupsAreJBraidgroups}:

\begin{proof}[Proof of Theorem \ref{NecklaceGroupsAreJBraidgroups}]
Combining Theorem \ref{LinkGroup} with Proposition \ref{ArtinToric}, a presentation for $\mathcal L_*^*(n,m)$ is given by 
\begin{subequations}
\label{NMPres1} 
\begin{align}
&(1) \,\, \mathrm{Generators}\!:\,  \{x_1,\dots,x_{n+2}\};\notag\\
&(2) \,\, \mathrm{Relations}\!: \,\notag\\
  & x_i\delta x_{n+2}\delta^q=\delta x_{n+2}\delta^qx_{n-r+i} \, \, \forall 1\leq i\leq r ,\label{NMPres1:1}\\
       &  x_{i}\delta x_{n+2}\delta^{q-1}=\delta x_{n+2}\delta^{q-1}x_{i-r} \, \, \forall r+1\leq i\leq n,\label{NMPres1:2}\\
      &   x_{n+1}\delta x_{n+2}\delta^{q-1}x_{n-r+1}\cdots x_n=\delta x_{n+2}\delta^{q-1}x_{n-r+1}\cdots x_nx_{n+1},\label{NMPres1:3}\\
      & x_{n+2}\delta=\delta x_{n+2}, \label{NMPres1:4}
\end{align}
\end{subequations}
where $\delta$ denotes $x_1\cdots x_{n+1}$.\\
Thus, by Remark \ref{AutoFree}, a presentation for $\mathcal L_*^*(n,m)$ is obtained from applying $\rho_{\Delta_n}$ to both sides of every relation in Presentation \eqref{NMPres1}. Applying Lemma \ref{ArtinDelta} and using that $x_1\cdots x_{n+1}$ is fixed by $\Delta_n$, this gives

\begin{subequations}
\label{NMPres2} 
\begin{align}
&(1) \,\, \mathrm{Generators}\!:\,  \{x_1,\dots,x_{n+2}\};\notag\\
&(2) \,\, \mathrm{Relations}\!: \,\notag\\
    & {}^{(x_1\cdots x_{n-i})}x_{n-i+1}\delta x_{n+2}\delta^q=\delta x_{n+2}\delta^q{}^{(x_1\cdots x_{r-i})}x_{r-i+1} \, \, \forall 1\leq i\leq r,\label{NMPres2:1}\\
       &  {}^{(x_1\cdots x_{n-i})}x_{n-i+1}\delta x_{n+2}\delta^{q-1}=\delta x_{n+2}\delta^{q-1}{}^{(x_1\cdots x_{n-i+r})}x_{n-i+r+1} \, \, \forall r+1\leq i\leq n,\label{NMPres2:2}\\
      &   x_{n+1}\delta x_{n+2}\delta^{q-1}x_{1}\cdots x_r=\delta x_{n+2}\delta^{q-1}x_{1}\cdots x_rx_{n+1},\label{NMPres2:3}\\
      & x_{n+2}\delta=\delta x_{n+2}, \label{NMPres2:4}
\end{align}
\end{subequations}
where $\delta$ denotes $x_1\cdots x_{n+1}$.\\
Now, for all $j\in[n]$ we have 
\begin{equation}\label{ConjTodelta}
    \begin{aligned}
    {}^{(x_1\cdots x_{j-1})}x_j\delta&=x_1\cdots x_jx_j\cdots x_{n+1}.
    \end{aligned}
\end{equation}
By Equation \eqref{ConjTodelta}, for all $i\in [r]$ Relation \eqref{NMPres2:1} is equivalent to $$x_1\cdots x_{n-i+1}x_{n-i+1}\cdots x_{n+2}\delta^qx_1\cdots x_{r-i}=\delta x_{n+2}\delta^{q}x_1\dots x_{r-i+1},$$
which in turn is equivalent to 
\begin{equation}\label{Eq1ProofNM}
\begin{aligned}
x_{n-i+1}\dots x_{n+2}\delta^{q}x_1\cdots x_{r-i}=x_{n-i+2}\cdots x_{n+2}\delta^qx_1\cdots x_{r-i+1}.
\end{aligned}
\end{equation}
When $i$ ranges from $1$ to $r$, the tuple $(n-i+1,r-i,n-i+2,r-i+1)$ ranges from $(n,r-1,n+1,r)$ to $(n-r+1,0,n-r+2,1)$. Thus, the set of relations 
\eqref{NMPres2:1} is equivalent to the set of relations
\begin{equation}
   \begin{aligned}
   x_i\cdots x_{n+2}\delta^qx_1\cdots x_{i-n+r-1}=x_j\cdots x_{n+2}\delta^qx_1\cdots x_{j-n+r-1},\, \forall n-r+1\leq i<j\leq n+1.
   \end{aligned} 
\end{equation}
With very similar computations, one obtains that the set of relations
\eqref{NMPres2:2}
is equivalent to the set of relations
\begin{equation} %ICI POUR COLLER AVANT
    \begin{aligned}
    x_i\cdots x_{n+2}\delta^{q-1}x_1\cdots x_{i+r-1}=x_j\cdots x_{n+2}\delta^{q-1}x_1\cdots x_{j+r-1}, \, \forall 1\leq i<j\leq n-r+1.
    \end{aligned}
\end{equation}
Thus, relabelling $x_{n+1}$ by $y$ and $x_{n+2}$ by $z$, the group $\mathcal L_*^*(n,m)$ admits the following presentation:
     \begin{subequations}
\label{NMPres3} 
\begin{align}
&(1) \,\, \mathrm{Generators}\!:\,  \{x_1,\dots,x_{n},y,z\};\notag\\
&(2) \,\, \mathrm{Relations}\!: \,\notag\\
  &x_i\cdots x_{n}yz\delta^qx_1\cdots x_{i-n+r-1}=x_j\cdots x_{n}yz\delta^qx_1\cdots x_{j-n+r-1},\,(*), \label{NMPres3:1}\\
       &    x_i\cdots x_{n}yz\delta^{q-1}x_1\cdots x_{i+r-1}=x_j\cdots x_{n}yz\delta^{q-1}x_1\cdots x_{j+r-1}, \, (**)  ,\label{NMPres3:2}\\
      &   y\delta z\delta^{q-1}x_{1}\cdots x_r=\delta z\delta^{q-1}x_{1}\cdots x_ry,\label{NMPres3:3}\\
      & z\delta=\delta z, \label{NMPres3:4}\\
      &(*)\, \forall n-r+1\leq i<j\leq n+1\, (**)\, \forall 1\leq i<j\leq n-r+1 \notag
\end{align}
\end{subequations}
where $\delta$ denotes $x_1\cdots x_{n}y$.
Now, Relations \eqref{NMPres3:1} and \eqref{NMPres3:2} imply \eqref{NMPres3:3}. Indeed, we have
\begin{equation*}
    \begin{aligned}
    \delta z\delta^{q-1}x_1\cdots x_ry&=x_{n-r+1}\cdots x_ny\delta^{q-1}x_1\cdots x_ny \text{  by \eqref{NMPres3:2},}\\
    &=y\delta z\delta^{q-1}x_1\dots x_r \text{  by \eqref{NMPres3:1}.}
    \end{aligned}
\end{equation*}
Thus, Presentation \eqref{BraidPres} is a presentation for $\mathcal L_*^*(n,m)$. Moreover, by remark \ref{AutoFree} the generators of Presentation \eqref{NMPres3} are meridians. Finally, under the identification between the generators of Presentation \eqref{BraidPres} and the meridians of $\mathrm L_*^*(n,m)$, we see that $y$ corresponds to the meridian around the circle in the internal core of the torus and $z$ corresponds to the meridian around the circle at the external core of the torus. Thus combining Remark \ref{SquareBraid} and Proposition \ref{LinkQuotient}, we get that quotienting $\mathcal L_*^*(n,m)$ by the normal closure of the meridians around the component we remove to obtain the link $\mathrm L_*(n,m)$ (respectively $\mathrm L^*(n,m)$) from $\mathrm L_*^*(n,m)$ amounts to quotienting by the normal closure of $z$ (respectively $y$) in Presentation \eqref{BraidPres}. This shows points (ii)-(iv) of Theorem \ref{NecklaceGroupsAreJBraidgroups} and concludes the proof.
\end{proof}

\begin{remark}\label{SquareLinks}
The Square \eqref{SquareBraid} is also valid for $\mathcal L_*^*(n,m),\mathcal L_*(n,m),\mathcal L^*(n,m)$ and $\mathcal L(n,m)$.
\end{remark}
\begin{remark}
Theorem \ref{NecklaceGroupsAreJBraidgroups} extends the philosophy of \cite{Gobet Toric} from toric reflection groups (the $J$-reflection groups with one conjugacy class of reflecting hyperplanes) to all $J$-reflection groups, in particular to all irreducible rank two complex reflection groups. See Table \ref{TableBraids} for the correspondence between rank two complex braid groups and torus necklace groups.
\end{remark}

The correspondence of Table \ref{TableBraids} can also be derived as follows: Given an irreducible complex reflection group $W\subset \mathrm{GL}_n(\C)$, recall from \cite{BessisZariski} and \cite{BMR} that its associated complex braid group $\B_\C(W)$ can be constructed as the fundamental group of the complement of the zero locus of a polynomial $f_W\in \C[z_1,\dots,z_n]$, called the \textit{discriminant} of $W$. Under this identification, the braid reflections are the so-called \textit{generators of the monodromy} or \textit{meridians} around each irreducible component of $f_W^{-1}(0)$ in $\C^n$ (see \cite{BessisZariski}). Now, by \cite[Theorem 2.10]{Milnor} the inclusion $\iota_W:S^{2n-1}\backslash (S^{2n-1}\cap f_W^{-1}(0))\into \C^n\backslash f_W^{-1}(0)$ is a homotopy equivalence, which sends meridians to meridians. In particular, the induced isomorphism between homotopy groups sends meridians to meridians.\\

In \cite{Bannai}, Bannai computed one such polynomial for each irreducible complex reflection group of rank two, which we compile in Table \ref{TableBraids} as well. Inspecting \cite[\S 2]{Bannai}, we see that two irreducible rank two complex reflection groups share the same braid group if and only if their discriminant is the same. It turns out that given an irreducible complex reflection group of rank two, the discriminant of $W$ is always a polynomial whose zero locus defines the torus necklace corresponding to $\B_\C(W)$ in Table \ref{TableBraids} (see \cite[\S 3]{Budney} for the defining polynomials of torus necklaces):\\\\
$\bullet$ For the complex braid group corresponding to the torus necklace group $\mathcal L(p,q)$ the discriminant given by Bannai is either $f(z_1,z_2)=z_1^p-z_2^q$ or $f(z_1,z_2)=z_1^q-z_2^p$, which satisfies either $S^3\cap f^{-1}(0)=\mathrm L(p,q)$ or $S^3\cap f^{-1}(0)=\mathrm L(q,p)\sim \mathrm L(p,q)$.\\
$\bullet$ For the complex braid group corresponding to the torus necklace group $\mathcal L_*(1,q)$ the discriminant given by Bannai is $f(z_1,z_2)=z_1^{2q}-z_2^{2}$, thus $S^3\cap f^{-1}(0)=\mathrm L(2q,2)$. By \cite[Proposition 3.5]{Budney}, the links $\mathrm L_*(1,q)$ and $\mathrm L(2q,2)$ are isotopic.\\
$\bullet$ For the complex braid group corresponding to the torus necklace group  $\mathcal L_*^*(1,q)$ the discriminant given by Bannai is $f(z_1,z_2)=z_1(z_1^{2q}-z_2^{2})$, thus $S^3\cap f^{-1}(0)=\mathrm L_*(2q,2)$. By \cite[Proposition 3.5]{Budney}, the links $\mathrm L_*^*(1,q)$ and $\mathrm L_*(2q,2)$ are isotopic.\\
$\bullet$ For the complex braid group $\B_\C(G_{13})$, which corresponds to the torus necklace group $\mathcal L_*(2,3)$, the discriminant given by Bannai is $f(z_1,z_2)=z_1(z_1^2-z_2^3)$, which satisfies $S^3\cap f^{-1}(0)=\mathrm L_*(2,3)$.\\
$\bullet$ Finally, for the complex braid group $\B_\C(G(d(2l+1),(2l+1),2))$ which corresponds to the torus necklace group $\mathcal L^*(2,2l+1)$, the discriminant given by Bannai is $f(z_1,z_2)=z_1(z_1^{2l+1}-z_2^2)$, thus $S^3\cap f^{-1}(0)=\mathrm L_*(2l+1,2)\sim\mathrm L^*(2,2l+1)$.\\

We obtain the following corollary:

\begin{corollary}\label{CorollaryHomotopyEquivalence}
Let $W$ be an irreducible complex reflection group of rank two and let $\mathrm L$ be the torus necklace corresponding to $\B_\C(W)$ in Table \ref{TableBraids}. There exists a homotopy equivalence between the regular orbit space of $W$ and $S^3\backslash \mathrm L$ which induces a bijection between (homotopy classes of) braid reflections and meridians. 
\end{corollary}

\begin{table}[hbt]
\centering
 \renewcommand{\arraystretch}{1.25}
\begin{tabular}{|l|c|c|}
\hline
Complex braid groups &  torus necklace groups & Discriminant \\
\hline
$\B_\C(G(3,3,2)),\B_\C(G_4),\B_\C(G_8),\B_\C(G_{16})$ &  $\mathcal L(2,3)$ & $z_1^3-z_2^2$
\\
\hline
 $\B_\C(G_{20})$ &  $\mathcal L(2,5)$& $z_1^5-z_2^2$\\
\hline
 $\B_\C(G(2l+1,2l+1,2))$ &  $\mathcal L(2,2l+1)$ & $z_1^{2l+1}-z_2^{2}$
\\
\hline
 $\B_\C(G_{12})$&  $\mathcal L(3,4)$ & $z_1^3-z_2^4$
\\
\hline
 $\B_\C(G_{22})$ &  $\mathcal L(3,5)$ & $z_1^3-z_2^5$\\
\hline
 $\B_\C(G_5),\B_\C(G_{10}),\B_\C(G_{18})$ &  $\mathcal L_*(1,2)$& $z_1^4-z_2^2$
\\
\hline
$\B_\C(G(d(2l+1),2l+1,2))$ &  $\mathcal L^*(2,2l+1)$& $z_1(z_1^{2l+1}-z_2^2)$
\\
\hline
 $\B_\C(G_6),\B_\C(G_9),\B_\C(G_{17})$ &  $\mathcal L_*(1,3)$& $ z_1^6-z_2^2$
\\
\hline
 $\B_\C(G_{13})$ & $\mathcal L_*(2,3)$& $z_1(z_1^2-z_2^3)$
\\
\hline
 $\B_\C(G_{14})$ &  $\mathcal L_*(1,4)$& $z_1^8-z_2^2$
\\
\hline 
 $\B_\C(G_{21})$ &  $\mathcal L_*(1,5)$& $z_1^{10}-z_2^{2}$
\\
\hline
$\B_\C(G(2d,2d,2))$ &  $\mathcal L_*(1,d)$& $z_1^{2d}-z_2^{2}$
\\
\hline
$\B_\C(G_7),\B_\C(G_{11}),\B_\C(G_{19})$ &  $\mathcal L_*^*(1,1)$& $z_1(z_1^2-z_2^2)$
\\
\hline
$\B_\C(G_{15})$ &  $\mathcal L_*^*(1,2)$& $z_1(z_1^4-z_2^2)$
\\
\hline
 $\B_\C(G(2cd,2c,2))$ &  $\mathcal L_*^*(1,d)$& $z_1(z_1^{2d}-z_2^2)$
\\
\hline
\end{tabular}
\caption{Rank two complex braid groups as torus necklace groups.}
\label{TableBraids}
\end{table}

\section{\texorpdfstring{Correspondence between generalised $J$-braid groups, torus necklace groups and circular groups}{}}
In Section 4, we showed that all $J$-braid groups are torus necklace groups. In particular, Theorem \ref{IsoBraidsJIgor} combined with Theorem \ref{NecklaceGroupsAreJBraidgroups} show that torus necklace groups with coprime parameters are isomorphic to circular groups, which in turn implies that they are Garside groups by Proposition \ref{CircularGarside}. The goal of this section is to extend this correspondence further by relaxing the coprimality assumption on the parameters.
More specifically, we define generalised $J$-braid groups and observe that the correspondence of Theorem \ref{NecklaceGroupsAreJBraidgroups} extends to all torus necklace groups instead of merely the case where $n$ and $m$ are coprime. Moreover, we show that up to abstract group isomorphisms, the family of generalised $J$-braid groups (equivalently, the family of torus necklace groups) and the family of circular groups are the same.

\subsection{\texorpdfstring{Generalised $J$-braid groups and torus necklace groups}{}}\label{SectionGeneralised}

When defining $J$-groups, one can remove the coprimality assumption on $n$ and $m$. Thus, one can extend the definition of $J$-reflection groups:

\begin{definition}
Let $k,b,n,c,m\in\N^*$ with $k,bn,cm\geq 2$. The group $W_b^c(k,bn,cm)$ is defined as $J\begin{pmatrix} k & bn & cm \\ & n & m\end{pmatrix}$. Such a group is called a \textbf{generalised $J$-reflection group}.
\end{definition}

\begin{remark}
The proof of Theorem \ref{GeneralPres} does not use the fact that $n$ and $m$ are coprime. Thus, Presentation \eqref{GeneralPresW} is a presentation of $W_b^c(k,bn,cm)$.
\end{remark}

We can now define $\B(W_b^c(k,bn,cm))$ in the exact same way:

\begin{definition}\label{DefBraidLinksClassical}
Let $W_b^c(k,bn,cm)$ be a generalised $J$-reflection group. Its braid group $\B_J(W_b^c(k,bn,cm))$ is defined precisely as in Definition \ref{DefBraid}. In this context, we still call \textbf{braid reflections} the conjugate of non trivial powers of elements of $\{x_1,\dots,x_n,y,z\}$.
\end{definition}
The notion of reflection isomorphism still holds for generalised $J$-reflection groups. In this case, to show that the analog of Theorem \ref{WellDefBraids} (ii) holds for generalised $J$-reflection groups, one would need to provide a precise classification of such groups up to reflection isomorphism. However, since these results are not necessary for the purposes of this article, we do not discuss them in detail and leave this discussion to the author's thesis (\cite{These}).

The isomorphisms of Theorem \ref{NecklaceGroupsAreJBraidgroups} are still valid, as we never used the fact that $n$ and $m$ are coprime in Section 4. In this context, the meridians of torus necklace groups correspond to braid reflections of generalised $J$-braid groups. Moreover, we have the following:
\begin{remark}\label{IsoNotCoprime}
By Proposition \ref{IsotopicLinks}, the braid groups of $W_b^c(k,bn,cm)$ and $W_c^b(k,cm,bn)$ are isomorphic. 
\end{remark}

\subsection{Link between torus necklace groups and circular groups}
First, all circular groups are torus necklace groups. Indeed, The link $L(n,m)$ is the $(n,m)$-torus link, whose link group is isomorphic to $\mathcal B(n,m)$ by Theorem \ref{NecklaceGroupsAreJBraidgroups}, which in turn is isomorphic to $G(n,m)$ by definition. In this case, the generators of Presentation \eqref{PresCircular} correspond to meridians of $\mathcal L(n,m)$ and braid reflections of $\mathcal B(n,m)$.
\begin{remark}\label{CircularIso}
The fact that the $(n,m)$-torus link group is isomorphic to the circular group $G(n,m)$ was stated by Bessis without proof in \cite{DualBessis} and the specific cases where either $n$ divides $m$ (see \cite{PicantinTorus}) or $n$ and $m$ are coprime (see \cite{Gobet Garside}) are already proven. The above discussion provides a proof of this fact. Moreover, Proposition \ref{IsotopicLinks} provides a new proof of the fact that the groups $G(n,m)$ and $G(m,n)$ are isomorphic. To the knowledge of the author, this fact was first proven in general in \cite[Proposition 2.20]{GarnierHoso}.
\end{remark}

\begin{RemConj}
It should be possible to generalise the isomorphism between the group with Presentations \eqref{PresTorusKnot} and $G(n,m)$ when $n$ and $m$ are coprime as proven in \cite{Gobet Garside} to the group with Presentation \eqref{PresTorusLinkRZ} and $G(n,m)$ for all $n,m\in\N^*$. More precisely, we conjecture that the morphism from the group with Presentation \eqref{PresTorusLinkRZ} to the group with Presentation \eqref{PresCircular} sending $x$ to $a_1\cdots a_m$, $y$ to $a_1\cdots a_n$ and $m_i$ to $a_i$ for all $i\in[d]$ is an isomorphism, but it is not clear to us how to construct the inverse from these presentations. A possible strategy would be to mimic what is done in \cite[Proposition 4.12 and Remark 4.13]{Gobet Garside}. Note however that by Remark \ref{CircularIso}, for all $n,m\in\N^*$ the group with Presentation \eqref{PresTorusLinkRZ} is isomorphic to $G(n,m)$.
\end{RemConj}

In \cite[Theorem 1]{SeifertLinks}, Burde and Murasugi classify all links whose link group's center is non-trivial (equivalently, all Seifert links). They fall into four categories (see also \cite[Section 3]{BoyerSeifert} or \cite[Proposition 3.3]{Budney} ):\\
(a) The link is either of the form $L_*(n,n)$ (equivalently, of the form $L^*(n,n)$) or $L_*^*(n,m)$ with $n,m\in\N^*$.\\
(a') The link is of the form $K(k)$ with $k\in\N$.\\
(b) The link is of the form $L_*(n,m)$ (equivalently, of the form $L^*(n,m)$) with $n\neq m\in\N^*$.\\
(c) The link is of the form $L(n,m)$ with $n,m\in \N^*$.\\

In cases (a) and (a') the link group is (abstractly) isomorphic to $F_{k-1}\times\Z\cong G(k,k)$ where $k$ is the number of connected components of the link (see \cite{SeifertLinks}). As stated in Remark \ref{CircularIso}, in case (c) the link group is isomorphic to $G(n,m)$. We now show that in case (b), the link group is also isomorphic to a circular group:

\begin{theorem}\label{LnIso}
For $n,m\in \N^*$ with $d=n\wedge m$, write $m=dm'$. The group $\mathcal L_*(n,m)$ is isomorphic to $G(d+1,(d+1)m')$.
\end{theorem}
We thus classified the isomorphism type of all link groups with non-trivial center:
\begin{corollary}\label{GarsideLinkGroups} A group is isomorphic to a link group with non-trivial center if and only if it is isomorphic to a circular group. In particular, a link group is Garside if and only if it has a non-trivial center.
\end{corollary}

\begin{proof}
The first part is a consequence of \cite[Theorem 1]{SeifertLinks} combined with Remark \ref{CircularIso} and Theorem \ref{LnIso}.
For the second assertion, we know that a Garside group must have a non-trivial center (see \cite[Proposition 2.6]{Origines}). Hence, for a link group to be Garside, it has to be a Seifert link group, which is isomorphic to a circular group. The proof then follows from the fact that circular groups are Garside groups (see \cite[Example 5]{Origines}).
\end{proof}

\begin{remark}
The fact that torus link groups are Garside was proven in \cite{PicantinTorus}, where the first explicit bi-automatic structure was given on these groups.
\end{remark}

\begin{remark}
In cases $(a),(a')$ and $(b)$, the isomorphism between the link group and the circular group is an abstract group isomorphism, as opposed to the case $(c)$ where the generators of the circular group correspond to meridians as discussed in the beginning ot the section.
\end{remark}

In order to prove Theorem \ref{LnIso}, we first exhibit an explicit isomorphism between $\mathcal L_*^*(n,m)$ and $G(d+2,d+2)$ which generalises the isomorphism constructed in \cite[Theorem 2.18]{BraidsJIgor}.
This will provide another presentation of $\mathcal L_*(n,m)$ which will be the one used to find an isomorphism with $G(d+1,(d+1)m')$, generalising the isomorphism constructed in \cite[Theorem 3.27]{BraidsJIgor}.\\

For the following definitions and results, we use the notations of Presentation \eqref{NMPres3}.
\begin{definition}\label{DefwWDelta}
Let $n,m\in\N^*$ and write $m=qn+r$ with $0\leq q$ and $0\leq r\leq n-1$. Write $d=n\wedge m=n\wedge r$ so that $m=dm'$, $n=dn'$ and $r=dr'$ with $n'\wedge m'=n'\wedge r'=1$.
Define $w\in \mathcal L_*^*(n,m)$ to be the element represented by one of the words in $$\{z\delta^qx_1\cdots x_r\}\cup \{x_i\cdots x_nyz\delta^{q-1}x_1\cdots x_{i+r-1}\}_{i\in [n-r+1]},$$ which is well defined by \eqref{NMPres3:2} and \eqref{NMPres3:4}. Moreover, define $W$ to be the element $wy\in \mathcal L_*^*(n,m)$ (equivalently, using \eqref{NMPres3:1}, the element represented by one of the words in $\{x_i\cdots x_nyz\delta^qx_1\cdots x_{i+r-n-1}\}_{i\in \llbracket n-r+1,n+1\rrbracket}$).\\
Finally, define $\Delta$ to be $w^{n'-r'}W^{r'}$.
\end{definition}

The first goal of this section is to show the following result:
\begin{proposition}\label{DeltaCentral}
Let $n,m\in\N^*$. The element $\Delta\in \mathcal L_*^*(n,m)$ defined as in Definition \ref{DefwWDelta} is central.
\end{proposition}

In order to prove Proposition \ref{DeltaCentral}, we need the following two lemmas, whose proof are identical to that of \cite[Lemmas 3.9 and 3.10]{BraidsJIgor} since the fact that $n$ and $m$ are coprime is not used in their proof:

\begin{lemma}[{\cite[Lemma 3.9]{BraidsJIgor}}]\label{xvw}
Let $n,m\in\N^*$ and $i\in [n-r]$, where $m=qn+r$ with $0\leq q$ and $0\leq r\leq n-1$. Let $w$ be as in Definition \ref{DefwWDelta}. Then $x_iw=wx_{i+r}$. Moreover, we have $yw=wy$.
\end{lemma}

%%%%%%%%%%%%%%%%%%%%%%%%%%%%%%%%%%%%%%%%%%%%%%%%%%%%%%%%%%%%
%%%%%%%%%%%%%%%%%%%%%%%%%%%%%%%%%%%%%%%%%%%%%%%%%%%%%%%%%%%%
%%%%%%%%%%%%%%%%%%%%%%%%%%%%%%%%%%%%%%%%%%%%%%%%%%%%%%%%%%%%%%%%%%%%%%%%%%%%%%%%%%%%%%%%%%%%%%%%%%%%%%%%%%%%%%%%%%%%%%%%%%%%%%%%%%%%%%%%%%%%%%%%%%%%%%%%%%%%%%%%%%%%%%%%%%%%%%%%%%%%
%%%%%%%%%%%%%%%%%%%%%%%%%%%%%%%%%%%%%%%%%%%%%%%%%%%%%%%%%%%%
%%%%%%%%%%%%%%%%%%%%%%%%%%%%%%%%%%%%%%%%%%%%%%%%%%%%%%%%%%%%%%%%%%%%%%%%%%%%%%%%%%%%%%%%%%%%%%%%%%%%%%%%%%%%%%%%%%%%%%%%
%%%%%%%%%%%%%%%%%%%%%%%%%%%%%%%%%%%%%%%%%%%%%%%%%%%%%%%%%%%%
%%%%%%%%%%%%%%%%%%%%%%%%%%%%%%%%%%%%%%%%%%%%%%%%%%%%%%%%%%%%
%%%%%%%%%%%%%%%%%%%%%%%%%%%%%%%%%%%%%%%%%%%%%%%%%%%%%%%%%%%%%%%%%%%%%%%%%%%%%%%%%%%%%%%%%%%%%%%%%%%%%%%%%%%%%%%%%%%%%%%%%%%%%%%%%%%%%%%%%%%%%%%%%%%%%%%%%%%%%%%%%%%%%%%%%%%%%%%%%%%%
%%%%%%%%%%%%%%%%%%%%%%%%%%%%%%%%%%%%%%%%%%%%%%%%%%%%%%%%%%%%
%%%%%%%%%%%%%%%%%%%%%%%%%%%%%%%%%%%%%%%%%%%%%%%%%%%%%%%%%%%%%%%%%%%%%%%%%%%%%%%%%%%%%%%%%%%%%%%%%%%%%%%%%%%%%%%%%%%%%%%%%%%%%%%%%%%%%%%%%%%%%%%%%%%%%%%%%%%%%%%%%%%%%%%%%%%%%%%%%%%%
%%%%%%%%%%%%%%%%%%%%%%%%%%%%%%%%%%%%%%%%%%%%%%%%%%%%%%%%%%%%
%%%%%%%%%%%%%%%%%%%%%%%%%%%%%%%%%%%%%%%%%%%%%%%%%%%%%%%%%%%%%%%%%%%%%%%%%%%%%%%%%%%%%%%%%%%%%%%%%%%%%%%%%%%%%%%%%%%%%%%%%%%%%%%%%%%%%%%%%%%%%%%%%%%%%%%%%%%%%%%%%%%%%%%%%%%%%%%%%%%%
%%%%%%%%%%%%%%%%%%%%%%%%%%%%%%%%%%%%%%%%%%%%%%%%%%%%%%%%%%%%
%%%%%%%%%%%%%%%%%%%%%%%%%%%%%%%%%%%%%%%%%%%%%%%%%%%%%%%%%%%%%%%%%%%%%%%%%%%%%%%%%%%%%%%%%%%%%%%%%%%%%%%%%%%%%%%%%%%%%%%%%%%%%%%%%%%%%%%%%%%%%%%%%%%%%%%%%%%%%%%%%%%%%%%%%%%%%%%%%%%%
\begin{lemma}[{\cite[Lemma 3.10]{BraidsJIgor}}]\label{xvW}
Let $n,m\in\N^*$ and $k\in [r]$, where $m=qn+r$ with $0\leq q$ and $0\leq r\leq n-1$. Let $W$ be as in Definition \ref{DefwWDelta}. Then $x_{n-r+k}W=Wx_{k}$. Moreover, we have $yW=Wy$.
\end{lemma}

%%%%%%%%%%%%%%%%%%%%%%%%%%%%%%%%%%%%%%%%%%%%%%%%%%%%%%%%%%%%
%%%%%%%%%%%%%%%%%%%%%%%%%%%%%%%%%%%%%%%%%%%%%%%%%%%%%%%%%%%%%%%%%%%%%%%%%%%%%%%%%%%%%%%%%%%%%%%%%%%%%%%%%%%%%%%%%%%%%%%%%%%%%%%%%%%%%%%%%%%%%%%%%%%%%%%%%%%%%%%%%%%%%%%%%%%%%%%%%%%%
%%%%%%%%%%%%%%%%%%%%%%%%%%%%%%%%%%%%%%%%%%%%%%%%%%%%%%%%%%%%
%%%%%%%%%%%%%%%%%%%%%%%%%%%%%%%%%%%%%%%%%%%%%%%%%%%%%%%%%%%%%%%%%%%%%%%%%%%%%%%%%%%%%%%%%%%%%%%%%%%%%%%%%%%%%%%%%%%%%%%%
%%%%%%%%%%%%%%%%%%%%%%%%%%%%%%%%%%%%%%%%%%%%%%%%%%%%%%%%%%%%
%%%%%%%%%%%%%%%%%%%%%%%%%%%%%%%%%%%%%%%%%%%%%%%%%%%%%%%%%%%%
\begin{remark}\label{Wr=0}
If $n=m=1$, we have $r=0$ so that the statement of Lemma \ref{xvW} is vacuously true.
\end{remark}

\begin{lemma}\label{DeltaEqProduct}
Let $n,m\in\N^*$ and let $\Delta$ be as in Definition \ref{DefwWDelta}. Then $\Delta=\delta^{m'}z^{n'}$.
\end{lemma}

\begin{proof}
In this proof, we construct a word and show that this word both represents $\Delta$ and $\delta^{m'}z^{n'}$ in $\mathcal L_*^*(n,m)$.
For $l\in \llbracket0,n'-1\rrbracket$ define the word $w_l\in \{x_1,\dots,x_n,y,z,\delta,\delta^{-1}\}^*$ inductively as $w_0=x_1\cdots x_nyz\delta^{q}\delta^{-1}x_1\cdots x_r$ and 
\begin{equation*}w_{l+1}=\begin{cases}
    w_lx_{j+1}\cdots x_nyz\delta^{q}\delta^{-1}x_1\cdots x_{j+r} \, \text{if $w_l$ ends by $x_j$, \,$j\in [n-r-1]$,}\\
    w_lx_{n-r+1}\cdots x_nyz\delta^q, \, \text{if $w_l$ ends with $x_{n-r}$},
    \\
    w_lx_{j+1}\cdots x_nyz\delta^qx_1\cdots x_{j+r-n} \, \text{if $w_l$ ends by $x_j$, \, $j\in \llbracket n-r+1,n-1\rrbracket$}.
\end{cases}\end{equation*}
Note that with $w$ and $W$ defined as in Definition \ref{DefwWDelta}, the element represented by $w_{l+1}$ in $\mathcal L_*^*(n,m)$ is equal to $w_lw$ if $w_l$ ends with $x_j$, $j\in [n-r-1]$ and is equal to $w_lW$ otherwise.\\
\fbox{The word $w_{n'-1}$ represents $\Delta$ in $\mathcal L_*^*(n,m)$:}
Since indices are taken modulo $n$, if the last letter of $w_l$ is $x_j$ with $j\neq n-r$, the last letter of $w_{l+1}$ is $x_{j+r}$. Observe now that $r$ has order $n'$ in $\Z/n$, which shows that the only element $l$ of $\llbracket 0,n'-2\rrbracket$ such that $(l+1)r\equiv n-r \, (\text{mod} \, n)$ is $n'-2$. In particular, as the last letter of $w_0$ is $x_r$, we see that the last letter of $w_l$ is $x_{(l+1)r}$ for all $l=0,\dots,n-2$ and the last letter of $w_{n-1}$ is $\delta$. Moreover, since $n\wedge r=n\wedge m=d$, the subgroup of $\Z/n$ generated by $r$ is the subgroup of $\Z/n$ generated by $d$ so that the set of last letters of the words $w_0,\dots,w_{n'-2}$ is $\{x_{(l+1)d}\}_{l\in \llbracket 0,n'-2\rrbracket}=\{x_d,x_{2d}\dots,x_{(n'-1)d}\}$.
This shows that amongst the last letters of $w_0,\dots,w_{n-2}$, we have $|\{x_d,x_{2d},\dots,x_{(n'-r'-1)d}\}|=n'-r'-1$ elements belonging to $\{x_1,\dots,x_{n-r-1}\}$ and $|\{x_{(n'-r')d},\dots,x_{(n'-1)d}\}|=r'$ elements belonging to $\{x_{n-r},\dots,x_{n-1}\}$. Thus, for $l\in \llbracket 0,n'-2\rrbracket$ there are exactly $r'$ times where $w_{l+1}$ represents $w_lW$ in $\mathcal L_*^*(n,m)$ and $n'-r'-1$ times where $w_{l+1}$ represents $w_lw$ in $\mathcal L_*^*(n,m)$.\\
Using Lemma \ref{xvw}, the elements $w$ and $W=wy$ commute so that $w_{n'-1}$ represents $w_0w^{n'-r'-1}W^{r'}=w^{n'-r'}W^{r'}=\Delta$ since $w_0$ represents $w$ by Definition \ref{DefwWDelta}. It remains to show that $w_{n'-1}$ represents $\delta^{m'}z^{n'}$ in $\mathcal L_*^*(n,m)$.\\
\fbox{The word $w_{n'-1}$ represents $\delta^{m'}z^{n'}$ in $\mathcal L_*^*(n,m)$:}
Observe that whenever $w_l$ ends with $x_j$, $j\in [n-r-1]$, the word length (in terms of the alphabet $\{x_1,\dots,x_n,y,z,\delta,\delta^{-1}\}$) of $w_{l+1}$ is equal to that of $w_l$ plus $n+r+q+3$. Otherwise, it is equal to that of $w_l$ plus $r+q+2$. Since the length of $w_0$ is $n+q+r+3$, we get that $w_{n'-1}$ is a word starting by $x_1$ of length 
\begin{equation}\label{Length1}
    \ell(w_{n-1})=(n+r+q+3)+(n'-r'-1)(n+r+q+3)+r'(r+q+2)=n'q+3n'+nn'-r',
\end{equation}
with exactly $n'$ occurrences of $z$, each of which comes after a $y$ and before a power of $\delta$. Removing from $w_{n'-1}$ every instance of the letters $\delta,\delta^{-1}$ and $z$ leaves a word $\tilde w_{n'-1}$ consisting of a succession of the word $x_1\cdots x_ny$. For any word $m$ and letter $x$, denote by $|m|_x$ the number of instances of $x$ in $m$. With this notation, we have
\begin{equation}\label{Length2}
|w_{n'-1}|_{\delta}=q(n'-r')+qr'=qn',\, |w_{n'-1}|_{\delta^{-1}}=n'-r',\, |w_{n'-1}|_z=(n'-r')+r'=n'.
\end{equation}
Combining equations \eqref{Length1} and \eqref{Length2}, the word $\tilde w_{n'-1}$ obtained from $w_{n'-1}$ after removing every instances of the letters $\delta,\delta^{-1}$ and $z$ has length 
\begin{equation}\label{Length3}
n'q+3n'+nn'-r'-qn'-(n'-r')-n'=n'(n+1).    
\end{equation}
Since $\ell(x_1\cdots x_ny)=n+1$, the word $\tilde w_{n'-1}$ is a succession of $n'$ copies of $x_1\cdots x_ny$. Finally, by $\eqref{BraidPresDef:1}$ the elements $z$ and $\delta=x_1\cdots x_ny$ commute so that the element represented by $w_{n'-1}$ in $\mathcal L_*^*(n,m)$ is $\delta^{n'+qn'-(n'-r')}z^{n'}=\delta^{qn'+r'}z^{n'}=\delta^{m'}z^{n'}$ since $qn'+r'=m'$. This concludes the proof.
\end{proof}

\begin{proof}[Proof of Proposition \ref{DeltaCentral}]
The proof is the same as that of \cite[Proposition 3.8]{BraidsJIgor}, with the only exception coming from the fact that we use that for all $i\in[n]$, we have $x_{i+n'r}=x_i$ instead of merely $x_{i+nr}=x_i$.
\end{proof}

\begin{remark}[Notation]\label{Remark/Notation}
The image of $\Delta$ under any of the quotients in Square \eqref{SquareLinks} is central by Proposition \ref{DeltaCentral}. In any of the groups $\mathcal L_*(n,m)$, $\mathcal L^*(n,m)$ and $\mathcal L(n,m)$, we still call $\Delta$ the resulting element in the quotient. Moreover, using the fact that $\Delta=\delta^{m'}z^{n'}$ we have the following:\\
(i) In $\B_*(n,m)$, the element $\Delta$ is $(x_1\cdots x_ny)^{m'}$.\\
(ii) In $\B^*(n,m)$, the element $\Delta$ is $(x_1\cdots x_n)^{m'}z^{n'}$.\\
(iii) In $\B(n,m)$, the element $\Delta$ is $(x_1\cdots x_n)^{m'}$. Note that in this case, the element $\Delta$ generates the center of $G(n,m)$ (see \cite[Corollary 2.11]{GarnierHoso}).
\end{remark}

We have the following result:
\begin{lemma}\label{SmallGen}
The $d+2$ elements $x_1,\dots,x_d,x_{d+1}\cdots x_ny$ and $z$ generate $\mathcal L_*^*(n,m)$.
\end{lemma}

For all $i\in[n]$ we introduce the element $X_i:=x_1\cdots \hat x_i\cdots x_ny\in \mathcal L_*^*(n,m)$, with the convention that $X_{i+n}=X_i$ for all $i\in \Z$. 

\begin{lemma}\label{GenerateXi}
The set $\{X_1,\dots,X_n,z,x_1\}\subset \mathcal L_*^*(n,m)$ generates $\mathcal L_*^*(n,m)$.
\end{lemma}

\begin{proof}
We show by induction on $i\in [n]$ that $\{x_1,\dots,x_i\}\subset \langle X_1,\dots,X_n,x_1\rangle$.\\
For $i=1$, this is trivial. Assume then that $\{x_1,\dots,x_i\}\subset\langle X_1,\dots,X_n,x_1\rangle$ for some $i\in [n-1]$. Then we have 
\begin{equation*}
    \begin{aligned}
        x_{i+1}&=(x_1\cdots x_i)^{-1}(x_1\cdots x_ny)(x_{i+2}\cdots x_ny)^{-1}(x_1\cdots x_i)^{-1}(x_1\cdots x_i)\\&=(x_1\cdots x_i)^{-1}x_1X_1X_{i+1}^{-1}(x_1\cdots x_i),
    \end{aligned}
\end{equation*}
which shows that $x_{i+1}\in \langle x_1,\dots,x_i,X_1,X_{i+1}\rangle\subset\langle X_1,\dots,X_n,x_1\rangle$ by induction hypothesis.\\
We get that $\{x_1,\dots,x_n,z\}\subset \langle X_1,\dots,X_n,z,x_1\rangle$. Since $y=(x_2\cdots x_n)^{-1}X_1$, this shows that 
\begin{equation*}\langle X_1,\dots,X_n,z,x_1\rangle=\mathcal L_*^*(n,m),\end{equation*} which concludes the proof.
\end{proof}

%%%%%%%%%%%%%%%%%%%%%%%%%%%%%%%%%%%%%%%%%%%%%%%%%%%%%%%%%%%%%%%%%%%%%%%%%%%%%%%%%%%%%%%%%%%%%%%%%%%%%%%%%%%%%%%%%%%%%%%%%%%
%%%%%%%%%%%%%%%%%%%%%%%%%%%%%%%%%%%%%%%%%%%%%%%%%%%%%%%%%%%%
%%%%%%%%%%%%%%%%%%%%%%%%%%%%%%%%%%%%%%%%%%%%%%%%%%%%%%%%%%%%%%%%%%%%%%%%%%%%%%%%%%%%%%%%%%%%%%%%%%%%%%%%%%%%%%%%%%%%%%%%%%%%%%%%%%%%%%%%%%%%%%%%%%%%%%%%%%%%%%%%%%%%%%%%%%%%%%%%%%%%
%%%%%%%%%%%%%%%%%%%%%%%%%%%%%%%%%%%%%%%%%%%%%%%%%%%%%%%%%%%%
%%%%%%%%%%%%%%%%%%%%%%%%%%%%%%%%%%%%%%%%%%%%%%%%%%%%%%%%%%%%%%%%%%%%%%%%%%%%%%%%%%%%%%%%%%%%%%%%%%%%%%%%%%%%%%%%%%%%%%%%%%%%%%%%%%%%%%%%%%%%%%%%%%%%%%%%%%%%%%%%%%%%%%%%%%%%%%%%%%%%
%%%%%%%%%%%%%%%%%%%%%%%%%%%%%%%%%%%%%%%%%%%%%%%%%%%%%%%%%%%%
\begin{proof}[Proof of Lemma \ref{SmallGen}]
Using Lemma \ref{GenerateXi}, it is enough to show that $\{X_1,\dots,X_n\}\subset \langle x_1,\dots,x_d,x_{d+1}\cdots x_ny,z\rangle$ to conclude the proof.\\
First, one sees that $\delta=x_1\cdots x_d(x_{d+1}\cdots x_ny)\in \langle x_1,\dots,x_d,x_{d+1}\cdots x_ny,z\rangle$. Now, observe that for all $i\in [n-r]$, the relation 
\begin{equation*}
    x_{i+1}\cdots x_nyz\delta^{q-1}x_1\cdots x_{i+r}=x_i\cdots x_nyz\delta^{q-1}x_1\cdots x_{i+r-1}
\end{equation*}
coming from \eqref{BraidPresDef:2} yields 
\begin{equation*}
\begin{aligned} &(x_1\cdots x_{i-1})(x_{i+1}\cdots x_nyz\delta^{q-1}x_1\cdots x_{i+r})(x_{i+r-1}\cdots x_ny)\\ &= (x_1\cdots x_{i-1})(x_i\cdots x_nyz\delta^{q-1}x_1\cdots x_{i+r-1})(x_{i+r-1}\cdots x_ny).
\end{aligned}\end{equation*}
Using \eqref{BraidPresDef:1}, the right hand side is equal to $z\delta^qX_{i+r}$, hence we get $X_iz\delta^q=z\delta^qX_{i+r}$. Similarly, for $i\in \llbracket n-r+1,n\rrbracket$ the relation 
\begin{equation*}
    x_{i+1}\cdots x_nyz\delta^qx_1\cdots x_{i+r-n}=x_i\cdots x_nyz\delta^qx_1\cdots x_{i+r-n-1}
\end{equation*}
coming from \eqref{BraidPresDef:3} yields $X_iz\delta^{q+1}=z\delta^{q+1}X_{i+r-n}$.\\Since $X_{i+r-n}=X_{i+r}$ and $\{z\delta^q,z\delta^{q+1}\}\subset \langle x_1,\dots,x_d,x_{d+1}\cdots x_ny,z\rangle,$ we see that for all $i\in [n]$, the element $X_i$ belongs to $\langle x_1,\dots,x_d,x_{d+1}\cdots x_ny,z\rangle$ if and only if $X_{i+r}$ does. Since $n\wedge r=d$, we have $\{X_1,\dots,X_n\}\subset \langle x_1,\dots,x_d,x_{d+1}\cdots x_ny,z\rangle$ if and only if $\{X_1,\dots,X_d\}\subset \langle x_1,\dots,x_d,x_{d+1}\cdots x_ny,z\rangle$, which holds since for all $i\in[d]$ we have $X_i=x_1\cdots x_{i-1}x_{i+1}\cdots x_d(x_{d+1}\cdots x_ny)$. This concludes the proof.
\end{proof}
In order to define the isomorphisms, recall the following from \cite{BraidsJIgor}:

\begin{definition}
Let $n,m\in \N^*$ be two coprime integers. If $m\neq 1$, define $n_{(m)}$ to be the unique integer in $[m-1]$ such that $\overline{nn_{(m)}}=1\in (\Z/m)^\times$. If $m=1$, define $n_{(m)}$ to be 1.
\end{definition}

\begin{lemma}[{\cite[Lemma 3.23]{BraidsJIgor}}]\label{CoprimeLemma}
Let $n,m\in \N^*$ be two coprime integers. Then we have $$nn_{(m)}+mm_{(n)}=nm+1.$$
\end{lemma}

\begin{definition}
Let $a,b\in\N^*$. Define $[a]_b$ as the unique integer in $\{1,\dots,b\}$ such that $a\equiv [a]_b \, (\mathrm{mod} \, b)$.
\end{definition}

Recall that given $n,m\in\N^*$, the element $\alpha\in G(n,m)$ is defined as $a_1\cdots a_n$.
\begin{proposition}\label{LtoG(d+2,d+2)}
Let $n,m\in \N^*$ with $d=n\wedge m$ and write $m=qn+r$ with $q\geq 0$ and $0\leq r\leq n-1$. Also write $m=dm'$ and $n=dn'$. The map
\[
\begin{array}{rcl}
\{x_1,\dots,x_n,y,z\}&\xto\varphi &G(d+2,d+2) \\ 
x_i & \mapsto &a_1^{\lceil\frac id\rceil-1}a_{[i]_d+2}a_1^{1-\lceil\frac id\rceil}\\
y & \mapsto & a_1^{n'}\alpha^{m'_{(n')}-n'}\\
z&\mapsto & a_2^{m'}\alpha^{n'_{(m')}-m'}
\end{array}
\]
extends to a morphism $\mathcal L_*^*(n,m)\xto\varphi G(d+2,d+2)$.
\end{proposition}

\begin{proof}
It is enough to show that the images of $x_1,\dots,x_n,y$ and $z$ by $\varphi$ satisfy the relations of Presentation \eqref{BraidPres}.  First, observe that for all $i\in [n]$, writing $i=ld+[i]_d$ we have 
\begin{equation}\label{eq1BtoG(d+2,d+2)}
\begin{aligned}
\varphi(x_1)\cdots\varphi(x_i)&=(\prod_{k=0}^{l-1}\prod_{j=1}^{d}\varphi(x_{kd+j}))\varphi(x_{ld+1})\cdots\varphi(x_{ld+[j]_d})\\
&=(\prod_{k=0}^{l-1}a_1^k(a_3\cdots a_{d+2})a_1^{-k})a_1^{l}(a_3\cdots a_{[i]_d+2})a_1^{-l}\\
&=(a_3\dots a_{d+2}a_1)^la_3\cdots a_{[i]_d+2}a_1^{-l},
\end{aligned}
\end{equation}
Similarly, we have
\begin{equation}\label{eq2BtoG(d+2,d+2)}
\begin{aligned}
\varphi(x_i)\cdots\varphi(x_n)&=\varphi(x_i)\cdots\varphi(x_{(l+1)d})(\prod_{k=l+1}^{n'-1}\prod_{j=1}^{d}\varphi(x_{kd+j}))\\
&=a_1^la_{[i]_d+2}\cdots a_{d+2}a_1^{-l}(\prod_{k=l+1}^{n'-1}a_1^k(a_3\cdots a_{d+2})a_1^{-k})\\
&=a_1^la_{[i]_d+2}\cdots a_{d+2}a_1(a_3\dots a_{d+2}a_1)^{n'-l-1}a_1^{-n'}.
\end{aligned}
\end{equation}

\noindent
\fbox{$x_1\cdots x_nyz=zx_1\cdots x_ny$:}\\
We have 
\begin{equation*}
\varphi(x_1)\cdots\varphi(x_n)\varphi(y)\underset{\eqref{eq1BtoG(d+2,d+2)}}=(a_3\cdots a_{d+2}a_1)^{n'}a_1^{-n'}a_1^{n'}\alpha^{m'_{(n')}-n'}
=(a_3\cdots a_{d+2}a_1)^{n'}\alpha^{m'_{(n')}-n'}.
\end{equation*}
From now on, by $\varphi(\delta)$ we mean $\varphi(x_1)\cdots \varphi(x_n)\varphi(y)$, which by the above then reads 
\begin{equation}\label{ImDelta}
\varphi(\delta)=(a_3\cdots a_{d+2}a_1)^{n'}\alpha^{m'_{(n')}-n'}.
\end{equation}
Since $a_2a_3\cdots a_{d+2}a_1=a_3\cdots a_{d+2}a_1a_2$ we have that $a_2$ commutes with $a_3\cdots a_{d+2}a_1$. Moreover, the element $\alpha$ is central so that $\varphi(\delta)$ commutes with $a_2$ but again since $\alpha$ is central, we get that $\varphi(\delta)$ commutes with $a_2^{m'}\alpha^{n'_{(m')}-m'}=\varphi(z)$. This shows that the relation $x_1\cdots x_nyz=zx_1\cdots x_ny$ is respected by $\varphi$. \\

\noindent
\fbox{$x_{i+1}\cdots x_{n}yz\delta^{q-1}x_1\cdots x_{i+r}=x_i\cdots x_{n}yz\delta^{q-1}x_1\cdots x_{i+r-1}$, $i\in [n-r]$:}\\
We show that the image by $\varphi$ of $x_i\cdots x_nyz\delta^{q-1}x_1\cdots x_{i+r-1}$ is the same for all $i\in [n-r+1]$. \\
In the following calculations, we use the fact that $\alpha$ is central without making explicit mention of it. Moreover, write $r=dr'$.\\
For all $i\in [n-r+1]$, writing $i=ld+[i]_d$ we have $i+r-1=(l+r')d+[i]_d-1$. Note that if $[i]_d=1$ then $i+r-1=(l+r')d$ and we still have
\begin{equation*}
\begin{aligned}
\varphi(x_1)\cdots\varphi(x_{i+r-1})&=(a_3\cdots a_{d+2}a_1)^{l+r'}a_1^{-l-r'}\\
&=(a_3\cdots a_{d+2}a_1)^{l+r'}a_3\cdots a_{[i]_d+1}a_1^{-l-r'},
\end{aligned}
\end{equation*}
where $a_3\cdots a_{[i]_d+1}=a_3\cdots a_2=1$ by convention. For all $i\in[n-r+1]$ we have:
\begin{equation}\label{LeftProd1(d+2,d+2)}
\begin{aligned}
&\varphi(x_{i})\cdots \varphi(x_n)\varphi(y)\varphi(z)\\
&\underset{\eqref{eq2BtoG(d+2,d+2)}}{=}a_1^la_{[i]_d+2}\cdots a_{d+2}a_1(a_3\cdots a_{d+2}a_1)^{n'-l-1}a_1^{-n'}a_1^{n'}\alpha^{m'_{(n')}-n'}a_2^{m'}\alpha^{n'_{(m')}-m'}\\
&=a_1^la_{[i]_d+2}\cdots a_{d+2}a_1(a_3\cdots a_{d+2}a_1)^{n'-l-1}a_2^{m'}\alpha^{m'_{(n')}+n'_{(m')}-m'-n'}.
\end{aligned}
\end{equation}
Similarly, we have
\begin{equation}\label{RightProd1(d+2,d+2)}
\begin{aligned}
&\varphi(\delta)^{q-1}\varphi(x_1)\cdots \varphi(x_{i+r-1})\\
&\underset{\eqref{eq1BtoG(d+2,d+2)}}{=}\varphi(\delta)^{q-1}(a_3\dots a_{d+2}a_1)^{l+r'}a_3\cdots a_{[i]_d+1}a_1^{-l-r'} \\
&\underset{\eqref{ImDelta}}{=}(a_3\cdots a_{d+2}a_1)^{n'(q-1)+l+r'}\alpha^{(m'_{(n')}-n')(q-1)}a_3\cdots a_{[i]_d+1}a_1^{-l-r'} \\
&=(a_3\cdots a_{d+2}a_1)^{m'-n'+l}\alpha^{(m'_{(n')}-n')(q-1)}a_3\cdots a_{[i]_d+1}a_1^{-l-r'} \text{  since $m'=qn'+r'$.} \\
\end{aligned}
\end{equation}
Since $a_2$ commutes with $a_3\cdots a_{d+2}a_1$ and $\alpha$ is central,  Equations \eqref{LeftProd1(d+2,d+2)} and \eqref{RightProd1(d+2,d+2)} imply that 
\begin{equation*}
    \begin{aligned}
        &\varphi(x_i)\cdots \varphi(x_n)\varphi(y)\varphi(z)\varphi(\delta)^{q-1}\varphi(x_1)\cdots \varphi(x_{i+r-1})\\
        &=a_1^la_{[i]_d+2}\cdots a_{d+2}a_1(a_3\cdots a_{d+2}a_1)^{m'-1}a_2^{m'}\alpha^{q(m'_{(n')}-n')+n'_{(m')}-m'}a_3\cdots a_{[i]_d+1}a_1^{-l-r'}\\
        &=a_1^la_{[i]_d+2}\cdots a_{d+2}a_1a_2\alpha^{q(m'_{(n')}-n')+n'_{(m')}-1}a_3\cdots a_{[i]_d+1}a_1^{-l-r'}\\
        &=a_1^la_{[i]_d+2}\cdots a_{d+2}a_1a_2a_3\cdots a_{[i]_d+1}\alpha^{q(m'_{(n')}-n')+n'_{(m')}-1}a_1^{-l-r'}\\
        &=a_1^{-r'}\alpha^{q(m'_{(n')}-n')+n'_{(m')}},
    \end{aligned}
\end{equation*}
which does not depend on $i$.\\ 
\noindent
\fbox{$x_{i+1}\cdots x_{n}yz\delta^{q}x_1\cdots x_{i+r-n}=x_i\cdots x_{n}yz\delta^{q}x_1\cdots x_{i+r-n-1}$, $i\in \llbracket n-r+1,n\rrbracket$:}\\
We show that the image by $\varphi$ of $x_{i+1}\cdots x_{n}yz\delta^{q}x_1\cdots x_{i+r-n}$ is the same for all $i\in \llbracket n-r,n\rrbracket$. \\
In the following calculations, we use the fact that $\alpha$ is central without making explicit mention of it. Moreover, write $r=dr'$.\\
For all $i\in \llbracket n-r,n\rrbracket$, writing $i=ld+[i]_d$ we have $i+r-n=(l+r'-n')d+[i]_d$. Note that for $i=n-r$ we have $i+r-n=0$ and we still have 
\begin{equation*}
    \begin{aligned}
        \varphi(x_1)\cdots \varphi(x_{i+r-n})&=(a_3\cdots a_{d+2}a_1)^0a_3\cdots a_2a_1^{-0}\\
        &=1.
    \end{aligned}
\end{equation*}
For all $i\in \llbracket n-r,n\rrbracket$ we have:
\begin{equation}\label{RightProd2(d+2,d+2)}
\begin{aligned}
&\varphi(\delta)^q\varphi(x_1)\cdots \varphi(x_{i+r-n})\\
&\underset{\eqref{eq1BtoG(d+2,d+2)}}{=}\varphi(\delta)^{q}(a_3\dots a_{d+2}a_1)^{l+r'-n'}a_3\cdots a_{[i]_d+1}a_1^{-l-r'+n'} \\
&\underset{\eqref{ImDelta}}{=}(a_3\cdots a_{d+2}a_1)^{qn'+l+r'-n'}\alpha^{q(m'_{(n')}-n')}a_3\cdots a_{[i]_d+1}a_1^{-l-r'+n'} \\
&=(a_3\cdots a_{d+2}a_1)^{m'+l-n'}\alpha^{q(m'_{(n')}-n')}a_3\cdots a_{[i]_d+1}a_1^{-l-r'+n'} \text{  since $m'=qn'+r'$.} \\
\end{aligned}
\end{equation}
Since $a_2$ commutes with $(a_3\cdots a_{d+2}a_1)$ and $\alpha$ is central,  Equations \eqref{LeftProd1(d+2,d+2)} and \eqref{RightProd2(d+2,d+2)} imply that 
\begin{equation*}
    \begin{aligned}
        &\varphi(x_{i+1})\cdots \varphi(x_n)\varphi(y)\varphi(z)\varphi(\delta)^{q}\varphi(x_1)\cdots \varphi(x_{i+r-n})\\
        &=a_1^la_{[i]_d+2}\cdots a_{d+2}a_1(a_3\cdots a_{d+2}a_1)^{m'-1}a_2^{m'}\alpha^{(q+1)(m'_{(n')}-n')+n'_{(m')}-m'}a_3\cdots a_{[i]_d+1}a_1^{-l-r'+n'}\\
        &=a_1^la_{[i]_d+2}\cdots a_{d+2}a_1a_2\alpha^{(q+1)(m'_{(n')}-n')+n'_{(m')}-1}a_3\cdots a_{[i]_d+1}a_1^{-l-r'+n'}\\
        &=a_1^la_{[i]_d+2}\cdots a_{d+2}a_1a_2a_3\cdots a_{[i]_d+1}\alpha^{(q+1)(m'_{(n')}-n')+n'_{(m')}-1}a_1^{-l-r'+n'}\\
        &=a_1^{n'-r'}\alpha^{(q+1)(m'_{(n')}-n')+n'_{(m')}},
    \end{aligned}
\end{equation*}
which does not depend on $i$. This concludes the proof.
\end{proof}

\begin{proposition}\label{G(d+2,d+2)toL}
Let $n,m\in \N^*$ with $d=n\wedge m$ and write $m=qn+r$ with $q\geq 0$ and $0\leq r\leq n-1$. Also write $m=dm'$ and $n=dn'$. The map 

\[
\begin{array}{rcl}
\{a_1,a_2,\dots,a_{d+2}\}&\xto\psi& \mathcal L_*^*(n,m) \\ 
a_1 & \mapsto &x_{d+1}\cdots x_nyz^{n'-m'_{(n')}}\delta^{n'_{(m')}-1}\\
a_2 &\mapsto &z^{m'_{(n')}}\delta^{m'-n'_{(m')}}\\
a_i&\mapsto & x_{i-2}\, \, \mathrm{for}\,\,i=3,\dots,d+2
\end{array}
\]
extends to a morphism $G(d+2,d+2)\xto\psi \mathcal L_*^*(n,m)$.
\end{proposition}

\begin{proof}
It is enough to show that the images of $a_1,a_2,\dots,a_{d+2}$ by $\psi$ satisfy the relations of Presentation \eqref{PresCircular}. We have
\begin{equation}\label{G(d+2,d+2)DeltaCenter}
    \begin{aligned}
        \varphi(a_2)\cdots \varphi(a_{d+2})\varphi(a_1)&=z^{m'_{(n')}}\delta^{m'-n'_{(m')}}\underbrace{x_1\cdots x_dx_{d+1}\cdots x_ny}_{=\delta}z^{n'-m'_{(n')}}\delta^{n'_{(m')}-1}\\
        &\underset{z\delta=\delta z}=\Delta \text{  by Lemma \ref{DeltaEqProduct}.}
    \end{aligned}
\end{equation}
By Lemma \ref{DeltaCentral}, the element $\varphi(a_2)\cdots \varphi(a_{d+2})\varphi(a_1)$ is central so that we have
\begin{equation*}
    \varphi(a_1)\cdots\varphi(a_{d+2})=\varphi(a_1)(\varphi(a_2)\cdots\varphi(a_{d+2})\varphi(a_1))\varphi(a_1)^{-1}=\varphi(a_2)\cdots\varphi(a_{d+2})\varphi(a_1).
\end{equation*}
Similarly, for all $i\in \llbracket 3,d+2\rrbracket$ we have
\begin{equation*}
    \begin{aligned}
        &\varphi(a_i)\cdots\varphi(a_{d+2})\varphi(a_1)\cdots\varphi(a_{i-1})\\&=(\varphi(a_2)\cdots\varphi(a_{i-1}))^{-1}\varphi(a_2)\cdots \varphi(a_{d+2})\varphi(a_1)(\varphi(a_2)\cdots \varphi(a_{i-1}))\\
        &=\varphi(a_2)\cdots\varphi(a_{d+2})\varphi(a_1).
    \end{aligned}
\end{equation*}
This concludes the proof.
\end{proof}

\begin{proposition}\label{LIsoG(d+2,d+2)}
The morphisms $\varphi$ and $\psi$ are inverse to each other.
\end{proposition}

\begin{proof}
Using Propositions \ref{LtoG(d+2,d+2)} and \ref{G(d+2,d+2)toL}, in order to show that $\varphi$ and $\psi$ are inverse to each other it is enough to show that $\varphi\circ\psi=Id_{G(d+2,d+2)}$ and that $\psi$ is surjective.\\

\noindent
\fbox{$\varphi\circ\psi=Id_{G(d+2,d+2)}$:}\\\\
\underline{$\varphi\circ \psi(a_1)=a_1$:}
We have 
\begin{equation}\label{Left1Iso(d+2,d+2)}
\begin{aligned}
\varphi(x_{d+1}\cdots x_ny)
&\underset{\eqref{eq2BtoG(d+2,d+2)}}{=}a_1(a_3\cdots a_{d+2}a_1)^{n'-1}a_1^{-n'}a_1^{n'}\alpha^{m'_{(n')}-n'}\\&=a_1(a_3\cdots a_{d+2}a_1)^{n'-1}\alpha^{m'_{(n')}-n'}.
\end{aligned}
\end{equation}

Similarly, using that $a_2$ commutes with $a_3\cdots a_{d+2}a_1$ and that $\alpha$ is central, we have
\begin{equation}\label{Right1Iso(d+2,d+2)}
    \begin{aligned}
        &\varphi(z^{n'-m'_{(n')}}\delta^{n'_{(m')}-1})\\
        &=a_2^{m'(n'-m'_{(n')})}\alpha^{(n'_{(m')}-m')(n'-m'_{(n')})}(a_3\cdots a_{d+2}a_1)^{n'(n'_{(m')}-1)}\alpha^{(m'_{(n')}-n')(n'_{(m')}-1)}\\
        &=a_2^{m'n'-m'm'_{(n')}}(a_3\cdots a_{d+2}a_1)^{n'(n'_{(m')}-1)}\alpha^{(n'_{(m')}-m'-n'_{(m')}+1)(n'-m'_{(n')})}\\
        &=a_2^{n'n'_{(m')}-1}(a_3\cdots a_{d+2}a_1)^{n'(n'_{(m')}-1)}\alpha^{(n'_{(m')}-m'-n'_{(m')}+1)(n'-m'_{(n')})} \text{  by Lemma \ref{CoprimeLemma}}\\
        &=(a_3\cdots a_{d+2}a_1)^{n'-1}\alpha^{(1-m')(n'-m'_{(n')})+n'n'_{(m')}-1}\\
        &=(a_3\cdots a_{d+2}a_1)^{n'-1}\alpha^{n'-m'_{(n')}} \text{  by Lemma \ref{CoprimeLemma}.}
    \end{aligned}
\end{equation}
Combining Equations \eqref{Left1Iso(d+2,d+2)} and \eqref{Right1Iso(d+2,d+2)}, we get 
\begin{equation*}
\begin{aligned}
    \varphi\circ\psi(a_1)&=\varphi(x_{d+1}\cdots x_nyz^{n'-m'_{(n')}}\delta^{n'_{(m')}-1})\\
    &=a_1(a_3\cdots a_{d+2}a_1)^{n'-1}\alpha^{m'_{(n')}-n'}(a_3\cdots a_{d+2}a_1)^{n'-1}\alpha^{n'-m'_{(n')}}\\
    &=a_1.
    \end{aligned}
\end{equation*}

\noindent
\underline{$\varphi\circ\psi(a_2)=a_2$:}
Using that $a_2$ commutes with $a_3\cdots a_{d+2}a_1$ and that $\alpha$ is central, we have
\begin{equation*}
\begin{aligned}
\varphi\circ\psi(a_2)&=\varphi(z^{m'_{(n')}}\delta^{m'-n'_{(m')}})\\
&=a_2^{m'm'_{(n')}}\alpha^{(n'_{(m')}-m')m'_{(n')}}(a_3\cdots a_{d+2}a_1)^{n'(m'-n'_{(m')})}\alpha^{(m'_{(n')}-n')(m'-n'_{(m')})}\\
        &=a_2^{m'm'_{(n')}}(a_3\cdots a_{d+2}a_1)^{n'(m'-n'_{(m')})}\alpha^{(n'_{(m')}-m')m'_{(n')}+(m'_{(n')}-n')(m'-n'_{(m')})}\\
        &=a_2^{1+n'(m'-n'_{(m')})}(a_3\cdots a_{d+2}a_1)^{n'(m'-n'_{(m')})}\alpha^{(n'_{(m')}-m)n'} \text{  by Lemma \ref{CoprimeLemma}}\\
        &=a_2
    \end{aligned}
\end{equation*}

\noindent
\underline{$\varphi\circ\psi(a_i)=a_i$ for all $i\geq 3$:}
For all $i=3,\dots,d+2$, we have $i-2=0\times d+(i-2)$ so that $\varphi\circ\psi(a_i)=\varphi(x_{i-2})=a_i$.
This shows that $\varphi\circ\psi=Id_{G(3,3)}$.\\

\noindent
\noindent
\fbox{$\psi$ is surjective:}\\\\
By Lemma \ref{SmallGen}, it is enough to show that $\{x_1,\dots,x_d,x_{d+1}\cdots x_ny,z\}\subset \im(\psi)$.\\\\
\noindent
\underline{$x_1,\dots,x_d\in \im(\psi):$} For all $i\in[d]$ we have $\psi(a_{i+2})=x_i$.\\

\noindent
\underline{$z\in \im(\psi)$:} Recall from the proof of Proposition \ref{G(d+2,d+2)toL} that $\psi(\alpha)=\Delta$. We thus have 
\begin{equation*}
\begin{aligned}
\psi(a_2^{m'}\alpha^{n'_{(m')}-m'}) &=(z^{m'_{(n')}}\delta^{m'-n'_{(m')}})^{m'}\Delta^{n'_{(m')}-m'}\\
&=(z^{m'_{(n')}}\delta^{m'-n'_{(m')}})^{m'}(\delta^{m'}z^{n'})^{n'_{(m')}-m'} \, \text{since $\Delta=\delta^{m'}z^{n'}$}\\
&\underset{z\delta=\delta z}{=}z^{m'm'_{(n')}+n'(n'_{(m')}-m')}\delta^{m'(m'-n'_{(m')})+m'(n'_{(m')}-m')}\\
&=z \, \,  \text{by Lemma \ref{CoprimeLemma} .}
\end{aligned}
\end{equation*}

\noindent
\underline{$x_{d+1}\cdots x_ny\in \im(\psi)$:} We have seen that $z\in \im(\psi)$. Moreover, by definition we have $z^{m'_{(n')}}\delta^{m'-n'_{(m')}}=\psi(a_2)$ and by Lemma \ref{DeltaEqProduct} we have $z^{n'}\delta^{m'}=\Delta=\psi(\alpha)$, which implies $\{\delta^{m'-n'_{(m')}},\delta^{m'}\}\subset\langle z,z^{m'_{(n')}}\delta^{m'-n'_{(m')}},z^{n'}\delta^{m'}\rangle\subset \im(\psi)$. In addition, since $n'_{(m')}$ is coprime with $m'$ we have that $m'$ and $m'-n'_{(m')}$ are coprime, hence $\delta\in \im(\psi)$.
Since $x_1,\dots,x_d\in \im(\psi)$, we get that $(x_1\cdots x_d)^{-1}\delta=x_{d+1}\cdots x_ny\in \im(\psi)$.\\
This concludes the proof.

\end{proof}

\begin{corollary}
For $n,m\in\N^*$ with $d=n\wedge m$, write $m=dm'$ and $n=dn'$. The groups $\mathcal L_*(n,m)\cong\mathcal L^*(m,n)$ admit the following presentation:
\begin{equation}\label{PresQuotientL1}
    \langle a_1,\dots a_{d+2}\,|\, a_1\cdots a_{d+2}=a_i\cdots a_{d+i+1},\, \forall i\in [d+2], \, a_2^{m'}=(a_1\cdots a_{d+2})^{m'-n'_{(m')}}\rangle,
\end{equation}
where indices are taken modulo $d+2$.
\end{corollary}

\begin{proof}
It follows from combining Propositions \ref{LtoG(d+2,d+2)} and \ref{LIsoG(d+2,d+2)} with Remark \ref{SquareLinks}, as quotienting by the normal closure of $y$ in $\mathcal L_*^*(n,m)$ amounts to quotienting by the normal closure of $\varphi(y)=a_2^{m'}(a_1\cdots a_{d+2})^{n'_{(m')}-m'}$ in $G(d+2,d+2)$.
\end{proof}
In order to show Theorem \ref{LnIso}, we need the following lemma, which is an adaptation of \cite[Lemma 3.31]{BraidsJIgor}:

\begin{lemma}\label{LemmaForIso2Conj1}
For $n,m\in \N^*$ with $d=n\wedge m$, write $n=dn'$ and $m=dm'$. Let $x,y\in \Z$ be such that $xn'_{(m')}+y(m'-n'_{(m')})=1$ (these exist since $m'\wedge n'_{(m')}=1$, hence $m'\wedge (m'-n'_{(m')})=1$). In Presentation \eqref{PresQuotientL1}, the equality
\begin{equation}\label{ustxy}
((a_3\cdots a_{d+2}a_1)^xa_2^y)^{m'}=a_1\cdots a_{d+2}
\end{equation}
holds. In particular, the element $((a_3\cdots a_{d+2}a_1)^{x}a_2^{y})^{m'}$ is central.
\end{lemma}

\begin{proof}
Since $a_1\cdots a_{d+2}=(a_3\cdots a_{d+2}a_1)a_2=a_2(a_3\cdots a_{d+2}a_1)$, the equality $a_2^{m'}=(a_1\cdots a_{d+2})^{m'-n'_{(m')}}$ implies that $a_2^{m'}=a_2^{m'-n'_{(m')}}(a_3\cdots a_{d+2}a_1)^{m'-n'_{(m')}}$, which in turn implies 
\begin{equation}\label{eq1LemmaForIso2Conj1}
    a_2^{n'_{(m')}}=(a_3\cdots a_{d+2}a_1)^{m'-n'_{(m')}}.
\end{equation}
Now, this implies that 

\begin{equation}\label{eq2LemmaForIso2Conj1}
\begin{aligned}
    (a_3\cdots a_{d+2}a_1)^{m'}&=(a_3\cdots a_{d+2}a_1)^{n'_{(m')}}(a_3\cdots a_{d+2}a_1)^{m'-n'_{(m')}}\\
    &\underset{\eqref{eq1LemmaForIso2Conj1}}{=}(a_3\cdots a_{d+2}a_1)^{n'_{(m')}}a_2^{n'_{(m')}}\\
    &=(a_3\cdots a_{d+2}a_1a_2)^{n'_{(m')}} \text{  since $a_3\cdots a_{d+2}a_1$ and $a_2$ commute.}
    \end{aligned}
\end{equation}
We thus get that
\begin{equation*}
    \begin{aligned}
        ((a_3\cdots a_{d+2}a_1)^{x}a_2^{y})^{m'}&=(a_3\dots a_{d+2}a_1)^{xm'}a_2^{ym'} \text{  since $a_3\cdots a_{d+2}a_1$ and $a_2$ commute}\\
        &\underset{\eqref{eq2LemmaForIso2Conj1}}{=}(a_3\cdots a_{d+2}a_1a_2)^{xn'_{(m')}}a_2^{ym'}\\
        &\underset{\eqref{eq1LemmaForIso2Conj1}}{=}(a_3\cdots a_{d+2}a_1a_2)^{xn'_{(m')}+y(m'-n'_{(m')})}\\
        &=a_1\cdots a_{d+2}  \text{  since $xn_{(m)}+y(m-n_{(m)})=1$.}
    \end{aligned}
\end{equation*}

\end{proof}

From now on, write $\{b_1,\dots,b_{d+1}\}$ for the generators of $G((d+1),(d+1)m')$. 

\begin{proposition}\label{Iso2Conj1}
For $n,m\in \N^*$ with $d=n\wedge m$, write $m=dm'$ and $n=dn'$.\\
Let $x,y\in \Z$ be such that $xn'_{(m')}+y(m'-n'_{(m')})=1$ (these exist since $m'\wedge n'_{(m')}=1$, hence $m'\wedge (m'-n'_{(m')})=1$). The map 
\[
\begin{array}{rcl}
\{b_1,\dots,b_{d+1}\}&\xto\varphi& \mathcal L_*(n,m) \\ 
b_1 & \mapsto & (a_3\cdots a_{d+2}a_1)^xa_2^y(a_4\cdots a_{d+2}a_1)^{-1}\\
b_i & \mapsto & a_{i+2} \,\,\mathrm{for}\,\, i=2,\dots,d\\
b_{d+1}& \mapsto & a_1
\end{array}
\]
extends to a morphism $G(d+1,(d+1)m')\to \mathcal L_*(n,m)$.
\end{proposition}

\begin{proof}
We show that for all $i\in[d+1]$, the value of $(\varphi(b_i)\cdots \varphi(b_{d+2})\varphi(b_1)\cdots\varphi(b_{i-1}))^{m'}$ does not depend on $i$.
We have
\begin{equation}\label{G(d+1,(d+1)m')DeltaCenter}
    \begin{aligned}
        (\varphi(b_1)\dots\varphi(b_{d+1}))^{m'}&=((a_3\cdots a_{d+2}a_1)^xa_2^y(a_4\cdots a_{d+2}a_1)^{-1}a_4\cdots a_{d+2}a_1)^{m'}\\
        &=a_1\cdots a_{d+2} \text{  by Lemma \ref{LemmaForIso2Conj1}}.
    \end{aligned}
\end{equation}
Since $(\varphi(b_1)\cdots \varphi(b_{d+1}))^{m'}=a_1\cdots a_{d+2}$ is central, for all $i=2,\dots,d+1$ we have
\begin{equation*}
\begin{aligned}
    &(\varphi(b_i)\cdots\varphi(b_{d+1})\varphi(b_1)\cdots\varphi(b_{i-1}))^{m'}\\&=(\varphi(b_1)\cdots\varphi(b_{i-1}))^{-1}(\varphi(b_1)\cdots\varphi(b_{d+1}))^{m'}(\varphi(b_1)\cdots\varphi(b_{i-1}))\\
    &=(\varphi(b_1)\cdots\varphi(b_{d+1}))^{m'},
    \end{aligned}
\end{equation*}
which concludes the proof.
\end{proof}

\begin{proposition}\label{Iso2Conj2}
For $n,m\in \N^*$ with $d=n\wedge m$, write $m=dm'$ and $n=dn'$. The map

\[
\begin{array}{rcl}
\{a_1,\dots,a_{d+2}\}&\xto\psi &G(d+1,(d+1)m')\\ 
a_1 & \mapsto &b_{d+1}\\
a_{2} & \mapsto & (b_1\cdots b_{d+1})^{m'-n'_{(m')}}\\
a_{3} &\mapsto & (b_1\cdots b_{d+1})^{n'_{(m')}}(b_2\cdots b_{d+1})^{-1}\\
a_i&\mapsto & b_{i-2} \,\,\mathrm{for}\,\, i=4,\dots,d+2
\end{array}
\]
extends to a morphism $\mathcal L_*(n,m)\to G(d+1,(d+1)m')$.
\end{proposition}
\begin{proof}
It is enough to show that the images of $a_1,a_2,\dots a_{d+2}$ by $\psi$ satisfy the relations of Presentation \eqref{PresQuotientL1}.\\

\noindent
\fbox{$a_1\cdots a_{d+2}=a_i\cdots a_{d+2}a_1\cdots a_{i-1}$,  $i=2,\dots,d+2$:}
We have 
\begin{equation}\label{ImPsiCentral}
    \begin{aligned}
     \psi(a_1)\cdots \psi(a_{d+2})&=b_{d+1}(b_1\cdots b_{d+1})^{m'-n'_{(m')}}(b_1\cdots b_{d+1})^{n'_{(m')}}(b_2\cdots b_{d+1})^{-1}b_2\cdots b_d\\
     &=(b_1\cdots b_{d+1})^{m'} \text{  since $(b_1\cdots b_d)^{m'}$ is central}.
    \end{aligned}
\end{equation}
Since $\psi(a_1)\cdots\psi(a_{d+2})=(b_1\cdots b_{d+1})^{m'}$ is central, for all $i=2,\dots,d+2$ we have 
\begin{equation*}
    \begin{aligned}
        &\psi(a_i)\cdots\psi(a_{d+2})\psi(a_1)\cdots\psi(a_{i-1})\\&=(\psi(a_1)\cdots\psi(a_{i-1}))^{-1}(\psi(a_1)\cdots\psi(a_{d+2}))(\psi(a_1)\cdots \psi(a_{i-1}))\\
        &=\psi(a_1)\cdots \psi(a_{d+2}).
    \end{aligned}
\end{equation*}

\noindent
\fbox{$a_2^{m'}=(a_1\cdots a_{d+2})^{m'-n'_{(m')}}$:}
We have
\begin{equation*}
    \begin{aligned}
        \psi(a_2)^{m'}\underset{\eqref{ImPsiCentral}}{=}(b_1\cdots b_{d+1})^{m'(m'-n'_{(m')})}=(\psi(a_1)\cdots \psi(a_{d+2}))^{m'-n'_{(m')}}.
    \end{aligned}
\end{equation*}
This concludes the proof.
\end{proof}

\begin{proof}[Proof of Theorem \ref{LnIso}]

Using Lemmas \ref{Iso2Conj1} and \ref{Iso2Conj2}, it is enough to show that $\varphi\circ\psi=Id_{\mathcal L_*(n,m)}$ and $\psi\circ\varphi=Id_{G((d+1),(d+1)m')}$.\\\\
\fbox{$\varphi\circ\psi=Id_{\mathcal L_*(n,m)}$:}\\\\
\underline{$\varphi\circ\psi(a_1)=a_1$:} We have $\varphi\circ\psi(a_1)=\varphi(b_{d+1})=a_1.$\\\\
\noindent
\underline{$\varphi\circ\psi(a_2)=a_2$:} We have
\begin{equation*}
\begin{aligned}
&\varphi\circ\psi(a_2)=\varphi((b_1\cdots b_{d+1})^{m'-n'_{(m')}})\\&=((a_3\cdots a_{d+2}a_1)^xa_2^y)^{m'-n'_{(m')}}\\
&=(a_3\cdots a_{d+2}a_1)^{x(m'-n'_{(m')})}a_2^{y(m'-n'_{(m')})} \text{  since $(a_3\cdots a_{d+2}a_1)a_2=a_2(a_3\cdots a_{d+2}a_1)$}\\&\underset{\eqref{eq1LemmaForIso2Conj1}}=a_2^{xn'_{(m')}}a_2^{y(m'-n'_{(m')})}=a_2 \, \, \text{since} \, \, xn'_{(m')}+y(m'-n'_{(m')})=1.
\end{aligned}
\end{equation*}

\noindent
\underline{$\varphi\circ\psi(a_3)=a_3$:} We have
\begin{equation*}
\begin{aligned}
&\varphi\circ\psi(a_3)=\varphi((b_1\cdots b_{d+1})^{n'_{(m')}}(b_2\cdots b_{d+1})^{-1})\\
&=((a_3\cdots a_{d+2}a_1)^xa_2^y)^{n'_{(m')}}(a_4\cdots a_{d+2}a_1)^{-1}\\
&=(a_3\cdots a_{d+2}a_1)^{xn'_{(m')}}a_2^{yn'_{(m')}}(a_4\cdots a_{d+2}a_1)^{-1} \text{  since $(a_3\cdots a_{d+2}a_1)a_2=a_2(a_3\cdots a_{d+2}a_1)$}\\
&\underset{\eqref{eq1LemmaForIso2Conj1}}=(a_3\cdots a_{d+2}a_1)^{xn'_{(m')}}(a_3\cdots a_{d+2}a_1)^{y(m'-n'_{(m')})}(a_4\cdots a_{d+2}a_1)^{-1}\,\\
&=a_3\, \, \text{since} \, \, xn'_{(m')}+y(m'-n'_{(m')})=1.
\end{aligned}
\end{equation*}

\noindent
\underline{$\varphi\circ\psi(a_i)=a_i$ for $i=4,\dots,d+2$:} For all $i\in\llbracket 4,d+2\rrbracket$ we have $i-2\in\llbracket 2,d\rrbracket$, which implies that $\varphi\circ\psi(a_i)=\varphi(b_{i-2})=a_i.$\\

\noindent
\fbox{$\psi\circ\varphi=Id_{G((d+1),(d+1)m')}$:}\\\\
\underline{$\psi\circ\varphi(b_1)=b_1$:} We have
\begin{equation*}
\begin{aligned}
\psi\circ\varphi(b_1)&=\psi((a_3\cdots a_{d+2}a_1)^xa_2^y(a_4\cdots a_{d+2}a_1)^{-1})\\&=(b_1\cdots b_{d+1})^{xn'_{(m')}}(b_1\cdots b_{d+1})^{y(m'-n'_{(m')})}(b_2\cdots b_{d+1})^{-1}
\\&=b_1 \, \, \text{since} \, \, xn'_{(m')}+y(m'-n'_{(m')})=1.
\end{aligned}
\end{equation*}

\noindent
\underline{$\psi\circ\varphi(b_i)=b_i$ for al $i=2,\dots,d+1$:} For all $i\in\llbracket 2,d\rrbracket$ we have $i+2\in\llbracket 4,d+2\rrbracket$, which implies that $\psi\circ\varphi(b_i)=\psi(a_{i+2})=b_i$. Finally, we have $\psi\circ\varphi(b_{d+1})=\psi(a_1)=b_{d+1}$. This concludes the proof.
\end{proof}

\begin{corollary}
The center of $\mathcal L_*^*(n,m), \mathcal L_*(n,m),\mathcal L^*(n,m)$ and $\mathcal L(n,m)$ is cyclic, generated by $\Delta$.
\end{corollary}

\begin{proof}
For this proof, write $\Delta_*^*$ for $\Delta\in \mathcal L_*^*(n,m)$ and $\Delta_*$ for its image in $\mathcal L_*(n,m)$.
By Equation \eqref{G(d+2,d+2)DeltaCenter}, the image of $\Delta_*^*$ by the isomorphism $\mathcal L_*^*(n,m)\to G(d+2,d+2)$ is $a_1\cdots a_{d+2}$, which generates the center by Theorem \ref{CenterCircularGroups}. Under this identification, by Equation \eqref{G(d+1,(d+1)m')DeltaCenter} the image of $\Delta_*$ by the isomorphism $\mathcal L_*(n,m)\to G((d+1),(d+1)m')$ is $(b_1\cdots b_{d+1})^{m'}$, which generates the center by Theorem \ref{CenterCircularGroups}. The exact same argument applies for $\mathcal L^*(n,m)$. Finally, the statement for $\mathcal L(n,m)$ is the content of Remark \ref{Remark/Notation} (iii).
\end{proof}
\begin{remark}
The fact that the center of a link group is cyclic is a standard fact in Knot Theory (see \cite[Theorem 1]{MurasugiCenter}).
\end{remark}

\end{document}